\documentclass[12pt,reqno]{amsart}
\usepackage{amsthm,amssymb,amsmath}
\usepackage[colorlinks=true, linkcolor=blue, urlcolor=blue, citecolor=blue]{hyperref}
\usepackage{xcolor}
\usepackage{mathtools}
\usepackage[all,cmtip]{xy}
\usepackage{tikz-cd}

\newtheorem{lemma}{Lemma}[section]
\newtheorem{theorem}[lemma]{Theorem}
\newtheorem{proposition}[lemma]{Proposition}
\newtheorem{corollary}[lemma]{Corollary}
\theoremstyle{definition}
\newtheorem{definition}[lemma]{Definition}
\newtheorem{remark}[lemma]{Remark}
\newtheorem{example}[lemma]{Example}
\newtheorem{convention}[lemma]{Convention}

\numberwithin{equation}{section}

\newcommand{\R}{\mathbb{R}}
\newcommand{\C}{\mathbb{C}}

\newcommand{\pr}{\mathbb{P}(E_{\ast})}
\newcommand{\ps}{\widetilde{p}^{\,\ast}_{\text{sp}}}

\newcommand{\rank}{\text{\rm rank}}

\newcommand{\pe}{\mathbb{P}(E_{\ast})}
\newcommand{\e}{E_{\ast}}
\newcommand{\G}{\textrm{GL}(r,\C)}
\newcommand{\eg}{E_{\G}}
\newcommand{\Oe}{\mathcal{O}_{\mathbb{P}(E_{\ast})}(1)}

\newcommand{\el}{\mathcal{L}}

\newcommand{\pf}{\mathbb{P}(E_{2{\ast}})}

\newcommand{\ef}{\mathbb{P}(E_{{1}{\ast}})\times_{X}\mathbb{P}(E_{{2}{\ast}})}
\newcommand{\tef}{\mathbb{P}(\widetilde{E}_{1})\times_{Y}\mathbb{P}(\widetilde{E}_{2})}

\newcommand{\w}{\widetilde}

\title[Projectivization of a parabolic vector bundle]{Positive Cones of the Projectivization of a parabolic vector bundle and Their Products over a Curve}

\author[A. Bansal]{Ashima Bansal}

\address{Department of Mathematics, Shiv Nadar University, NH91, Tehsil
Dadri, Greater Noida, Uttar Pradesh 201314, India}

\email{ashima.bansal@snu.edu.in}

\author[I. Biswas]{Indranil Biswas}

\address{Department of Mathematics, Shiv Nadar University, NH91, Tehsil
Dadri, Greater Noida, Uttar Pradesh 201314, India}

\email{indranil.biswas@snu.edu.in, indranil29@gmail.com}

\author[S. Majumder]{Souradeep Majumder}
\address{Indian Institute of Science Education and Research Tirupati, Srinivasapuram, Yerpedu Mandal, Tirupati Dist, Andhra Pradesh, India – 517619}
\email{souradeep@labs.iisertirupati.ac.in}

\subjclass[2010]{14C17, 14C20, 14H60, 14F05}
\keywords{Parabolic vector bundle, projectivization, nef cones, pseudo-effective cone}

\begin{document}

\begin{abstract}
We compute the positive cones of the projectivization of a parabolic vector bundle and the fiber product of two parabolic projective bundles over a smooth complex projective curve. Specifically, we determine their Néron--Severi groups and compute their nef and pseudoeffective cones. Moreover, for the projectivization of a parabolic vector bundle, we explicitly describe the generators of the higher nef and pseudoeffective cones. As an application, we obtain a necessary and sufficient criterion for the semistability of a parabolic vector bundle.
\end{abstract}

\maketitle

\section{Introduction}
The positivity condition of algebraic cycles plays a fundamental role in algebraic and birational geometry. In recent years, several notions of positivity for higher codimension cycles, such as nef and pseudoeffective cones, have been extensively studied. The nef cones of various smooth irreducible projective varieties have attracted considerable attention over the past few decades (see \cite[Section 1.5]{La}, \cite{Mi}, \cite{Ful}, \cite{KMR}). Knowledge of the generators of nef and pseudoeffective cones helps to identify the extremal directions of positivity, thereby controlling how subvarieties intersect and degenerate.

For projective bundles over curves, these cones admit explicit descriptions in terms of the Harder--Narasimhan filtration of the vector bundle. In \cite{Mi}, Miyaoka computed the nef cone of \(\mathbb{P}_{C}(E)\) for a vector bundle \(E\) over a smooth complex projective curve \(C\). Subsequently, \cite{Ful} generalized this to cycles of arbitrary codimension, showing that the effective cones of cycles and their duals on \(\mathbb{P}_{C}(E)\) are completely determined by the numerical data arising from the Harder--Narasimhan filtration of \(E\).

Since the generators of these cones for projective bundles over smooth projective curves have already been computed, it is natural to ask the same question for parabolic projective bundles, introduced in \cite{BL}. These provide important examples of finite group quotients (which need not be smooth varieties; see Example~\ref{singular}), because they arise as finite quotients of the projectivization of orbifold bundles. To date, relatively 
little is known about the nef and pseudoeffective cones of singular varieties. Thus, computing these 
positive cones in this context provides meaningful examples in the singular setting.

In Section 3, we give a complete description of the numerical groups, the pseudoeffective cones, and the nef cones of parabolic projective bundles over smooth projective curves. The quotient map from the projectivization of an orbifold bundle to the projectivization of the corresponding parabolic vector bundle need not be flat, and hence we do not have a flat pullback. Instead, the special pullback map --- defined in \cite{Fu} --- is used.


In Section 4, we study the positive cones of the fiber product of two parabolic projective bundles over a smooth projective curve. More precisely, given parabolic vector bundles $E_{1*}$ and $E_{2*}$ on $X$, we consider the fiber product
\begin{equation*}
S\,=\,\mathbb P(E_{1*})\,\times_X\, \mathbb P(E_{2*}),
\end{equation*}
and investigate the cones of effective cycles on $S$. First, certain intersection numbers of the natural divisor classes on $S$ are computed (see Corollary~\ref{int}). Then the pseudoeffective cones of cycles in both the semistable and unstable cases --- within a specified range --
are determined (see Theorem \ref{stable} and Theorem \ref{usp}); their generators are described in terms of the numerical invariants associated with the Harder--Narasimhan filtrations of the underlying parabolic vector bundles.

\section{Preliminaries}
This section recalls the numerical group of dimension $k$-cycles and its closed subcones. We also review the theory of parabolic bundles and the construction of parabolic projective bundles.

\subsection{Chow groups and numerical groups}

Throughout this subsection, $X$ is a complex projective variety of dimension $n$.

\subsubsection{Cycles and Chow groups}

Let $Z_{k}(X)$ denote the group of $k$--cycles on $X$ with coefficients in $\mathbb R$. Any subscheme $Z\,
\subset\, 
X$ of dimension $k$ has a fundamental cycle denoted by $[Z]\,\in\,{Z}_{k}(X)$ defined as in \cite[\S~1.5]{Fu}. 
To study the geometry of cycles on $X$, various equivalence relations have been 
introduced on $Z_{k}(X)$. One example is rational equivalence; see \cite[\S~1.3 and \S~1.6]{Fu}.
The Chow group $\textrm{CH}_{k}(X)$ is the quotient of ${Z}_{k}(X)$ modulo rational equivalence, which may
still have infinite rank. When $X$ is smooth, denote $\textrm{CH}^{k}(X) \,=\, \textrm{CH}_{n-k}(X)$.
There is a graded ring structure on $\textrm{CH}^{\ast}(X)\,:=\, \bigoplus_{k\,\geq\, 0} \textrm{CH}^{k}(X)$ \cite[\S~8]{Fu}.

Take a vector bundle $E$ on $X$. For every $k$ and $m$, there exists a linear map
$$\textrm{CH}_{m}(X)\ \, \xrightarrow{\,\,\,c_{k}(E)\,\cap\,_{-}\,\,\,}\ \,\textrm{CH}_{m-k}(X),$$
where $c_k(E)$ denotes the $k$-th Chern class \cite[\S~3.2]{Fu}.
For notational convenience, $c_{1}(E)\,\bigcap\,[X]$ will be denoted by $c_{1}(E)$. Furthermore, there is a natural homomorphism from the Picard group
\begin{equation*}
\textrm{Pic}(X)\ \longrightarrow \ \textrm{CH}_{n-1}(X)_{\mathbb{Z}},
\end{equation*}
defined by $\mathcal{L}\, \longmapsto \, c_{1}(\mathcal{L})\,\bigcap\, [X]$.
This homomorphism is injective if $X$ is a normal scheme, and it is an isomorphism if all 
local rings of $X$ are UFD, in other words, when $X$ is locally factorial.

\subsubsection{Chow ring of a finite group quotient}

Let $G$ be a finite group acting on a smooth projective variety $X$, and let $Y\,=\,X/G$ be the quotient variety. Denote by $\pi\,:\,X\,\longrightarrow\,Y$ the finite quotient map.
Let $\textrm{CH}_{\ast}(X)^{G}$ denote the ring of $G$-invariants of $\textrm{CH}_{\ast}(X)$. Then there is a canonical group isomorphism $\textrm{CH}_{\ast}(Y)\, =\, \textrm{CH}_{\ast}(X)^{G}$.
For any subvariety $W$ of $X$, let 
\begin{equation*}
I_{W}\,\,=\,\, \{g\,\in\,G\,\,\big\vert\,\, g_{|_W}\,=\, {\rm Id}_{W}\}    
\end{equation*}
be the inertia group, and let 
\begin{equation*}
e_{W}\,=\, \textrm{card}(I_{W})/\textrm{deg}_{i}(W/V),    
\end{equation*}
where $V\,=\, \pi(W)$, and deg$_{i}(W/V)$ is the degree of inseparability of $K(W)$ over $K(V)$, the
function fields of $W$ and $V$ respectively. For a subvariety $V$ of the $Y$ set
\begin{equation}\label{special}
\pi_{\textrm{sp}}^{\ast}[V]\ \,=\ \, \sum\, e_{W} [W],    
\end{equation}
where the sum over all irreducible components $W$ of $\pi^{-1}(V)$. This determines an isomorphism $Z_{\ast}(Y)\,=\, Z_{\ast}(X)^{G}$, and $\textrm{CH}_{\ast}(X)^{G}$ is the quotient of $Z_{\ast}(X)^{G}$ modulo the subspace generated by
\begin{equation*}
\left\{\sum_{g\,\in\,G}g_{\ast}[\textrm{div}(r)]\,\,\,\big\vert\,\,r\in\,K(W)^{\ast},\,\,\, W\,\subset\,X\right\}.    
\end{equation*}
Note that the composition of maps
\begin{equation}\label{comp}
\textrm{CH}_{\ast}(Y)\,\xlongrightarrow{\,\,\,\pi_{\textrm{sp}}^{\ast}\,\,\,}\, \textrm{CH}_{\ast}(X)^{G}\,\hookrightarrow\,\textrm{CH}_{\ast}(X)\,\xlongrightarrow{\,\,\,\pi_{\ast}\,\,\,}\, \textrm{CH}_{\ast}(Y)     
\end{equation}
is multiplication by card($G$). In addition, $\textrm{CH}_{\ast}(Y)$ may also be made into a ring. Indeed, in this case, one has an isomorphism
\begin{equation*}
 \textrm{CH}_{\ast}(Y)\ \,=\ \,\textrm{CH}_{\ast}(X)^{G},   
\end{equation*}
so $\textrm{CH}_{\ast}(Y)$ is the ring of $G$--invariants of $\textrm{CH}_{\ast}(X)$.

In fact, if $V,\,W$ are subvarieties of $Y$, one may construct a refined intersection class $V\cdot W$ in $\textrm{CH}_{\ast}(V\,\cap\,W)$ defined as 
\begin{equation*}
V\cdot W\ \,=\ \, (1/|G|)\eta_{\ast}(\pi_{\textrm{sp}}^{\ast}[V]\cdot \pi_{\textrm{sp}}^{\ast}[W]),   
\end{equation*}
where $\eta$ is the projection from $\pi^{-1}(V\,\cap\,W)$ to $V\,\cap\, W$. Note that $\pi_{\textrm{sp}}^{\ast}(a\cdot b)\,=\, \pi_{\textrm{sp}}^{\ast}(a)\cdot \pi_{\textrm{sp}}^{\ast}(b)$ and $\pi_{\ast}(\pi_{\textrm{sp}}^{\ast}(a)\cdot c)\,=\, a\cdot \pi_{\ast}(c)$, for cycles $a,\,b$ on $Y$ and cycles $c$ on $X$. 

The canonical homomorphism 
\begin{equation*}
\textrm{CH}^{\ast}(Y)\ \,\xlongrightarrow{\,\,\,\cap\,[Y]\,\,\,}\ \, \textrm{CH}_{\ast}(Y)    
\end{equation*}
is an isomorphism of rings. This shows in particular that the ring structure on $\textrm{CH}_{\ast}(Y)$ is independent of $X$, 
so produces a pull-back homomorphism for arbitrary morphisms of such varieties. See \cite[Examples 8.3.12 and 17.4.10]{Fu} for more details.

\subsubsection{Numerical equivalence}

To define numerical equivalence, we will work with an equivalence relation which is coarser than rational 
equivalence.

\textbf{Smooth case.}\ When $X$ is smooth, there is an intersection pairing
\begin{equation}\label{e1}
\textrm{CH}_{k}(X)\,\times\, \textrm{CH}^{k}(X) \, \longrightarrow \, \mathbb{R}
\end{equation}
determined by the ring structure on $\textrm{CH}^{\ast}(X)$ and the natural point counting degree function ${\rm 
deg}\, :\, \textrm{CH}_{0}(X)\, \longrightarrow \, \mathbb{R}$. The \textit{numerical group} $N_{k}(X)$ is the quotient 
of $\textrm{CH}_{k}(X)$ by the kernel of this pairing; denote $N^{k}(X) \,:=\, N_{n-k}(X)$. The pairing in \eqref{e1} 
induces a perfect pairing $N_{k}(X)\times N^{k}(X)\, \longrightarrow \, \mathbb{R}$, in particular, 
we have $N^{k}(X)\,\cong\, (N_{k}(X))^{\vee}$.

\textbf{Singular case.}\ When $X$ is singular, we do not have an intersection pairing. Instead,
the Chern class action can be used. Following \cite[\S~19]{Fu}, we say that a $k$--cycle $Z$ is
\textit{numerically trivial} --- denoted by $Z\,\equiv\, 0$ --- if 
\begin{equation*}
\textrm{deg}(P\cap [Z]_{Chow})\ =\ 0
\end{equation*}
for any weight $k$ polynomial $P$ in Chern classes of vector bundles on $X$; a Chern polynomial is
naturally seen as an operator on Chow groups using the linearity and commutativity of the action of
the Chern classes. The \textit{numerical group} is the quotient $$N_{k}(X)\, =\,\textrm{CH}_{k}(X)/\equiv .$$
It is a real vector space of finite dimension, and it is nonzero only
when $0\,\leq\, k\,\leq \,\dim\, X\,=\, n$. The class in ${N}_{k}(X)$ of a real $k$--cycle $Z$ is denoted
by $[Z]$. Clearly, both $N_{0}(X)$ and $N_{n}(X)$ are isomorphic to $\mathbb{R}$.

The \textit{dual numerical group} $N^{k}(X)\,:=\, (N_{k}(X))^{\vee}$ is no longer isomorphic to
$N_{n-k}(X)$. It can also be defined as follows:
\begin{equation}\label{upper numerical groups}
{N}^{k}(X)\ = \ \frac{\text{Homogeneous Chern }\, \mathbb{R}\text{--polynomials $P$ of
weight }\, k}{\text{Chern polynomials }\, P\, \text{ such that }\, P\,\cap\,
\alpha \,=\,0\, \text{ for all}\, \alpha\,\in\, N_{k}(X)}.
\end{equation}
The multiplication of polynomials induces a graded ring structure on $N^{\ast}(X)$. The action of Chern classes induces linear maps
\begin{equation}\label{ch}
N^{k}(X)\,\times\, N_{m}(X) \, \longrightarrow \, N_{m-k}(X)
\end{equation}
that we continue to denote by $P\,\cap\,\alpha$, or $P\cdot\alpha$. Hence, there exists a ``cyclification'' map
\begin{equation}\label{isomorphism}
N^{k}(X) \ \xlongrightarrow{\phi} \ N_{n-k}(X)
\end{equation}
defined by $P\,\longmapsto\,
P\,\cap\,[X]$, which is, in general, not an isomorphism. The cyclification $N^{1}(X) \, \longrightarrow \,
N_{n-1}(X)$ is injective (see
\cite[Example 19.3.3]{Fu}). Dually, $N^{n-1}(X) \, \longrightarrow \, N_{1}(X)$ is 
onto. More generally, $N_{\ast}(X)$ is a module over $N^{\ast}(X)$.

Note that ${N}^{1}(X)$ is the N\'eron-Severi group of real Cartier divisors modulo numerical equivalence. 
Thus, ${N}_{1}(X)$ is the space of curves with
$\R$--coefficients modulo classes that have vanishing intersections against the first Chern class
of invertible sheaves. The formal dual ${N}^{1}(X)$ is then the space of the real first Chern class of 
invertible sheaves modulo that have vanishing intersections against every curve.

\begin{remark}
When $X$ is smooth and projective, the intersection theory endows
$\textrm{CH}_{\ast}(X)$ with a ring structure graded by codimension \emph{\cite[Chapter 8]{Fu}}. This ring
structure descends to ${N}_{\ast}(X)$.
\end{remark}

\begin{remark}
The classical definition for numerical triviality on smooth varieties can be recovered as follows: A cycle 
$Z\,\in\, Z_{k}(X)$ is numerically trivial if $[Z]\cdot\beta\, =\, 0$ for
all $\beta\,\in\, {N}_{n-k}(X)$, where $[Z]$ denotes the numerical class of $Z$ in ${N}_{k}(X)$.
\end{remark}

\subsection{Positive cones}

\subsubsection{The pseudoeffective cone}
We say that a class $\alpha\,\in\, {N}_{k}(X)$ is \textit{effective} if $\alpha\, =\, [Z]$ for some effective 
cycle $Z$. This notion is closed under positive linear combinations, and hence it is natural to consider the 
following:

\begin{definition}
The closure of the convex cone generated by the effective $k$-cycles on $X$ in ${N}_{k}(X)$ is denoted 
$\overline{\textrm{Eff}}_{k}(X)$. It is called the \textit{pseudoeffective} cone. A class $\alpha\,\in\, 
{N}_{k}(X)$ is called \textit{pseudoeffective} if it lies in 
${\overline{\textrm{Eff}}}_{k}(X)$.
\end{definition}

Note that $\overline{\textrm{Eff}}_{1}(X)$ is referred to as the \textit{closed cone of curves}, which is also 
denoted by $\overline{\textrm{NE}}(X)$ in the literature.

\begin{definition}\label{dp}
We say that $\beta\,\in \,{N}^{k}(X)$ is \textit{pseudoeffective} if $\phi(\beta)\,\in\,
\overline{\textrm{Eff}}_{n-k}(X)$, where $\phi$ is the map in \eqref{isomorphism}. The pseudoeffective
dual classes form a closed cone in ${N}^{k}(X)$, and this cone is denoted by $\overline{\textrm{Eff}}^{\,k}(X)$.
\end{definition}

\subsubsection{The nef cone}

Recall the perfect pairing $N^{k}(X)\times N_{k}(X)\,\longrightarrow \,\mathbb{R}$ defined by $(P,\,
\alpha)\,\longmapsto\, P\,\cap\,\alpha$.

\begin{definition}
 The \textit{nef cone} $\textrm{Nef}^{\,\,k}(X)\,\subset\, N^{k}(X)$ is the dual of $\overline{\textrm{Eff}}_{k}(X)\,\subset\, N_{k}(X)$, through the perfect pairing above.
\end{definition}
By definition, nefness is preserved under proper pullbacks. If $X$ is smooth, a cycle $\alpha \,\in\, {N}^{k}(X)$ is nef if and only if $(\alpha\cdot \beta)\,\geq\, 0$ for all effective cycles $\beta$ of dimension $k$.

\subsection{Orbifold bundle}

Let $Y$ be a smooth complex projective variety of dimension $n$. Its group of algebraic automorphisms will be 
denoted by $\textrm{Aut}(Y).$ Let $\Gamma\, \subset\, \textrm{Aut}(Y)$ be a finite subgroup. So $\Gamma$
acts on $Y$ through algebraic automorphisms.

An \textit{orbifold bundle} on $Y$, with $\Gamma$ as the \textit{orbifold group}, is a vector bundle $V$ on 
$Y$ together with a lift of the action of $\Gamma$ on $Y$ to $V$, i.e., $\Gamma$ acts on the total space of 
$V$ such that the action of any $g\,\in \,\Gamma$ gives a vector bundle isomorphism between $V$ to 
$(g^{-1})^{\ast}V$. A subsheaf $F$ of an orbifold bundle $V$ is called an \textit{orbifold subsheaf} if 
the action of $\Gamma$ on $V$ preserves $F$.

Fix a polarization on $Y$ that is preserved by the action of $\Gamma$.
An orbifold bundle $V$ on $Y$ is called \textit{orbifold semistable} (respectively, \textit{orbifold 
stable}) if for any orbifold subsheaf $F$ of $V$, with $0\,<\,\textrm{rank(F)}\,<\,\textrm{rank(V)}$, the
following inequality holds:
\begin{equation*}
\frac{\textrm{deg}(F)}{\textrm{rank} (F)} \,\leq\, \frac{\textrm{deg} (V)}{\textrm{rank} (V)}\,\, \
\left(\textrm{respectively,}\,\, \frac{\textrm{deg} (F)}{\textrm{rank} (F)} \,<\,
\frac{\textrm{deg} (V)}{\textrm{rank} (V)}\right).
\end{equation*}

\begin{proposition}[{\cite[Lemma 2.7]{Bi1}}]\label{orbss}
An orbifold bundle $V$ is orbifold semistable if and only if it is semistable in the usual sense.
\end{proposition}

\subsection{Parabolic vector bundles} 

Let $X$ be an irreducible smooth projective variety defined over $\C$. A \textit{parabolic vector bundle} on $X$ with a divisor $D$ is a vector bundle $E$ on $X$ together with filtrations of its restrictions to the components of $D$, each equipped with strictly increasing weights in $[0,1)$ (see \cite{MY} and \cite{Se}).

The notions of subsheaf, quotient, direct sum, tensor product, dual, symmetric product, and exterior powers for vector bundles extend naturally to parabolic vector bundles. In addition, we have the notions of semistability 
and stability for parabolic vector bundles (see \cite{MY}, \cite{Bi4}, \cite{Yo}).

Every parabolic sheaf $E_{\ast}$ has a unique \textit{Harder-Narasimhan filtration}, i.e., there exists a unique filtration 
\begin{equation}\label{HN}
E\, =\, E^{0}\,\supsetneq\, E^{1}\, \supsetneq\, \cdots\,\supsetneq\, E^{l}\, \supsetneq\, E^{l+1}\, = \,0
\end{equation}
such that all $(E^i/E^{i+1})_{\ast}$, $0\, \leq\, i\, \leq\, l$, equipped with the induced
parabolic structure, are parabolic semistable and \textrm{par-}$\mu((E^{i}/E^{i+1})_{\ast})\, >\, $ \textrm{par-}$\mu((E^{i-1}/E^{i})_{\ast})$ for all $i\,\in\, \{1,\, 2,\, \cdots,\, l\}$. Let
\begin{equation}\label{eqs1}
Q^{i}_\ast\ := \ \left( \frac{E^i}{E^{i+1}} \right)_{\ast}
\end{equation}
be the quotient parabolic sheaf. Let
\begin{equation}\label{parrank}
r_i\, =\, \operatorname{rk}(Q^{i}_{\ast}), \quad d_i \,=\, \textrm{par-}\textrm{deg} \,(Q^{i}_{\ast}),\quad \mu_i\, :=\, \textrm{par-}\mu(Q^{i}_{\ast})\, =\, \frac{d_i}{r_i}
\end{equation}
respectively be the rank, degree and parabolic slope of $Q^{i}_{\ast}$.
Hence, we have
\begin{equation*}
\mu_{0}\ <\ \mu_{1}\ <\ \cdots\ <\ \mu_{l-1}\ <\ \mu_{l}.
\end{equation*}
\subsection{Correspondence between orbifold bundles and parabolic vector bundles }

Let $Y$ be a smooth projective variety equipped with a faithful action of a finite group $\Gamma$, and let
\begin{equation*}
p\ :\ Y \ \longrightarrow \ Y/\Gamma\ =:\ X
\end{equation*}
be the quotient map, such that $X$ is smooth. As described in \cite{Bi1}, there is an equivalence of categories between parabolic vector bundles on $X$ (with rational weights determined by the ramification data of $p$) and $\Gamma$--equivariant vector bundles on $Y$. In this correspondence, a parabolic vector bundle $E_{\ast}$ on $X$ corresponds to a $\Gamma$--bundle $\widetilde{E}$ on $Y$, and we have:
\begin{equation*}
\textrm{deg}(\widetilde{E})\ =\ |\Gamma|. \textrm{par-}\textrm{deg}(E_{\ast}) \qquad \text{ and } \qquad \textrm{rk}(E_{\ast})\ =\ \textrm{rk}(\widetilde{E}), 
\end{equation*}
where $|\Gamma|$ denotes the cardinality of the finite group $\Gamma$.

\begin{proposition}[{\cite[Lemma 3.16]{Bi1}}]\label{semistable}
The orbifold bundle
$\widetilde{E}$ is orbifold semistable if and only if $E_{\ast}$ is parabolic semistable.
\end{proposition}

{}From Proposition \ref{orbss} and Proposition \ref{semistable} it follows immediately that
$\widetilde{E}$ is semistable if and only if $E_{\ast}$ is parabolic semistable. In fact,
the above correspondence between orbifold vector bundles and parabolic vector bundles preserves
the Harder--Narasimhan filtration.

\subsection{Projectivization of parabolic vector bundle}\label{par_project}

Given a parabolic vector bundle $\e$, a parabolic analogue of the projective bundle is defined in \cite{BL} 
using the notion of ramified principal bundles (see \cite{BBN}). This construction will be briefly recalled.

Let $\e$ be a parabolic vector bundle over $X$ of rank $r$. Let 
\begin{equation*}
\phi\ :\ E_{\G} \ \longrightarrow \ X
\end{equation*}
be the corresponding ramified principal $\G$--bundle. Let $\mathbb{P}^{r-1}$ be the projective space parametrizing the hyperplanes in $\C^{r}$. The standard action of $\G$ on $\mathbb{C}^{r}$ produces an action of $\G$ on $\mathbb{P}^{r-1}$.
The \textit{projectivization} of $\e$, denoted by $\mathbb{P}(\e)$, is defined to be the associated (ramified) fiber bundle 
\begin{equation}\label{pe}
\pe \ := \ \eg (\mathbb{P}^{r-1})\ := \ \eg \times ^{\G} \mathbb{P}^{r-1}\, \longrightarrow \, X.
\end{equation}
In other words, $\pe$ is the quotient of $\eg \times \mathbb{P}^{r-1}$ where two points
$(z_1,\, b_1)$ and $(z_2,\, b_2)$, where $z_1,\, z_2\, \in\, \eg$ and $b_1,\, b_2\, \in\, \mathbb{P}^{r-1}$,
are identified if there is $A\, \in\, \G$ such that $z_2\,=\, z_1A$ and $b_2\,=\, A^{-1}(b_1)$.

Take any point $x\,\in\, D$ and any $z\,\in\, \phi^{-1}(x)$. Let $G_{z}\,\subset\, \G$ be the isotropy subgroup for $z$ for the action of $\G$ on $\eg$. We recall that $G_{z}$ is a finite group. Let $n_{x}$ be the order of $G_{z}$. Note that the order of the group $G_{z}$ is independent of the choice of $z\,\in\, \phi^{-1}(x)$ because $\G$ acts transitively on $\phi^{-1}(x)$. The number of distinct integers $n_{x}$ as $x$ varies over $D$ is finite. Let 
\begin{equation}\label{N(e)}
N(E_{\ast})\ \, =\ \, \textrm{l.c.m.} \{n_{x} \,\,\big\vert\,\,\, x\,\in\, D\}
\end{equation}
be the least common multiple of all these finitely many positive integers $n_{x}$. For any point $y\,\in\, 
\mathbb{P}^{r-1}$, let $H_{y}\,\subset\, \G$ be the isotropy subgroup for the natural action of $\G$ on 
$\mathbb{P}^{r-1}$; so $H_{y}$ is a maximal parabolic subgroup of $\G$. The group $H_{y}$ then acts on the 
fiber of the tautological line bundle $\mathcal{O}_{\mathbb{P}^{r-1}}(1) \, \longrightarrow \, \mathbb{P}^{r-1}$ over the point $y$. 
From the definition of $N(E_{\ast})$ in \eqref{N(e)} it follows immediately that for any $z\,\in\, \phi^{-1}(D)$, 
and any $y\, \in\, \mathbb{P}^{r-1}$, the subgroup $G_{z}\,\cap\, H_{y}\, \subset \, \G$ acts trivially on the 
fiber of the line bundle $\mathcal{O}_{\mathbb{P}^{r-1}}(N(E_{\ast}))\, :=\, 
\mathcal{O}_{\mathbb{P}^{r-1}}(1)^{\otimes N(E_{\ast})}$ over the point $y$. Consider the action of $\G$ on the 
total space of $\mathcal{O}_{\mathbb{P}^{r-1}}(N(E_{\ast}))$ constructed using the standard action of $\G$ on 
$\mathbb{C}^{r}$, and let
\begin{equation*}
E_{\G}(\mathcal{O}_{\mathbb{P}^{r-1}}(N(E_{\ast})))\, := \, \eg \times^{\G} 
\mathcal{O}_{\mathbb{P}^{r-1}}{(N(E_{\ast}))} \, \longrightarrow \, X
\end{equation*}
be the associated fiber 
bundle. As the natural projection $\mathcal{O}_{\mathbb{P}^{r-1}}(N(E_{\ast})) \, \longrightarrow \, 
\mathbb{P}^{r-1}$ intertwines the action of $\G$ on $\mathcal{O}_{\mathbb{P}^{r-1}}(N(E_{\ast}))$ and 
$\mathbb{P}^{r-1}$, it produces a projection
\begin{equation}\label{projection}
\eg(\mathcal{O}_{\mathbb{P}^{r-1}}(N(E_{\ast})))\, \longrightarrow \, \eg (\mathbb{P}^{r-1})\, = \, \pe
\end{equation}
over the map $p\, :\, Y\, \longrightarrow\, X$.

Using the observation above that $G_{z}\cap H_{y}$ acts trivially on the fiber of 
$\mathcal{O}_{\mathbb{P}^{r-1}} (N(E_{\ast}))$ over $y$, it follows immediately that the projection in 
\eqref{projection} makes $\eg(\mathcal{O}_{\mathbb{P}^{r-1}}(N(E_{\ast}))$ an algebraic line bundle over $\pe$;
it will be referred to as the tautological line bundle. Let
\begin{equation}\label{et}
\mathcal{O}_{\pe} (1)\, \longrightarrow\ \, \pe
\end{equation}
denote this tautological line bundle.

There is another description of $\pe$, arising from the correspondence between the ramified
$\textrm{GL}_{r}(\mathbb{C})$--bundles and parabolic vector bundles. The action of $\Gamma$ on
$\widetilde{E}$ produces a left action of $\Gamma$ on the corresponding projective
bundle $\mathbb{P}(\widetilde{E})$. The variety $\pe$ in \eqref{pe} is the quotient 
\begin{equation}\label{q}
\mathbb{P}(\widetilde{E})\,\ \xrightarrow{\,\,\,\,\widetilde{p}\,\,\,}\ \, 
\mathbb{P}(\widetilde{E})/\Gamma\,\ =\, \ \pe.
\end{equation}
Also, there is a natural isomorphism of line bundles 
\begin{equation*}
\mathcal{O}_{\mathbb{P}(\widetilde{E})}(N(E_{\ast}))/\Gamma\,\ =\,\ \mathcal{O}_{\pe} (1).
\end{equation*}
Hence, we have
\begin{equation}\label{taut}
\widetilde{p}^{\,\ast}(\Oe)\ =\ \mathcal{O}_{\mathbb{P}(\widetilde{E})}(N(E_{\ast}))
\end{equation}
where $\widetilde p$ is the quotient map in \eqref{q}.

From now on, to simplify the notation, set
\[
t\ :=\ N(E_*).
\]

In general, $\pr$ need not be smooth. Its smoothness depends on whether the components of the fixed
point locus --- for the action of $\Gamma$ on $\mathbb{P}(\widetilde{E})$ --- are divisors or not.

\begin{example}\label{singular}
Consider $\Gamma\,:=\,\mathbb{Z}/2\mathbb{Z}\,=\, \{1, -1\}\,\hookrightarrow\, \textrm{PGL}(2,\mathbb{C})$ acting on $\mathbb{P}^{1}$ defined as:
\begin{equation*}
\Gamma\,\times\, \mathbb{P}^{1}\,\longrightarrow\, \mathbb{P}^{1}
\end{equation*}
\begin{equation*}\label{action}
(-1,\, [a:b])\,\longrightarrow\, [-a:b].
\end{equation*}
Clearly, $[1:0]$ and $[0:1]$ are the two fixed points under this action, and $\mathbb{P}^{1}/\Gamma\,\cong\, \mathbb{P}^{1}$ is a smooth variety. Let $\widetilde{E}\,:=\, \mathcal{O}\,\oplus\, \mathcal{O}(-1)$ be a rank two vector bundle on $\mathbb{P}^{1}$, equipped with the natural lift of the action of $\Gamma$ defined above. One can see that $\mathbb{P}(\widetilde{E})$ has an isolated fixed point. Thus, $\mathbb{P}(\widetilde{E})/\Gamma$ is not a smooth variety. 
\end{example}

The above description of $\pr$ as the quotient $\mathbb{P}(\widetilde{E})/\Gamma$ will be used in what follows.

\section{Positive cones of parabolic projective bundles}

In this section, we explicitly compute generators of the numerical group and the pseudoeffective cone
of a parabolic projective bundle over a curve.

From now on, it is assumed that $X$ is a smooth projective curve with $D$ being a reduced effective divisor on $X$. Denote by $E_{\ast}$ a parabolic vector bundle on $X$, with a fixed parabolic structure along $D$, whose all the parabolic weights are rational numbers.

There exists a smooth projective curve $Y$, and a finite ramified Galois morphism 
\begin{equation*}
p\ :\ Y \ \longrightarrow \ X
\end{equation*}
with Galois group $\Gamma\, =\, \textrm{Gal}(K(Y)/K(X))$, 
such that $E_{\ast}$ corresponds to a unique orbifold vector
bundle $\widetilde{E}$ on $Y$ (see Section 2.5). As before, the Galois group $\textrm{Gal}(K(Y)/K(X))$ is
denoted by $\Gamma$.

By the construction of $\pr$ (see \eqref{q}), we have the following commutative diagram:
\begin{equation}\label{0}
\begin{tikzcd} 
\mathbb{P}(\widetilde{E}) \arrow[r, "\widetilde{\pi}" ] \arrow[d,"\widetilde{p}"]
& Y\arrow[d , swap,"p"] \\
\pr\arrow[r, "\pi"]
& |[, rotate=0]| X. 
\end{tikzcd}
\end{equation}

\begin{remark}
As noted in Example \ref{singular}, the variety $\pr$ need not be smooth; hence, the map $\widetilde{p}$ need not be flat. Therefore, the flat pullback is not defined. However, we have Fulton's special pullback (see §2.1.2)
\begin{equation*}
\widetilde{p}^{\,\ast}_{\textrm{sp}}\,:\,\textrm{CH}_{\ast}(\pr)\,\longrightarrow\, \textrm{CH}_{\ast}(\mathbb{P}(\widetilde{E})),
\end{equation*}
which is a ring homomorphism.
\end{remark}
\begin{convention}\label{con1}
Using the canonical isomorphism 
\begin{equation}\label{same}
\textrm{CH}^{\ast}(\pr) \,\xlongrightarrow{\;\cap\,[\pr]\;}\, \textrm{CH}_{\ast}(\pr),    
\end{equation}
for a vector bundle $\mathcal{F}$ on $\pr$,  it can be seen that for any $k \,\geq\, 1$, 
\begin{equation*}
c_{1}(\mathcal{F})^{k} \,\cap\, [\pr] \;=\; \big(c_{1}(\mathcal{F}) \,\cap\, [\pr]\big)^{k}
\end{equation*}
as elements of $\textrm{CH}_{\ast}(\pr)$.

In view of the isomorphism in \eqref{same}, we will not distinguish between 
$c_{1}(\mathcal{F}) \,\cap\, [\pr]$ and $c_{1}(\mathcal{F})$ in the subsequent sections. 
For further details, see \cite[Example~17.4.10]{Fu}.
\end{convention}

\subsection{Intersection theory on parabolic projective bundle}

In this subsection, we compute certain intersection numbers on $\pr$ using analogous intersection
numbers on $\mathbb{P}(\widetilde{E})$. Let $\widetilde{\mathcal{L}}$ denote the line bundle on $\mathbb{P}(\widetilde{E})$
corresponding to the fiber of $\widetilde{\pi}$ (see \eqref{0}) over a point $y\,\in\, Y$. It is well-known that

\begin{equation}\label{proj}
c_{1}(\widetilde{\mathcal{L}})^{2} \ = \ 0, \ \ \, c_{1}(\mathcal{O}_{\mathbb{P}(\widetilde{E})}(1))^{r-1}\cdot c_{1}
(\widetilde{\mathcal{L}})\ = \ [pt], \ \ \, c_{1}(\mathcal{O}_{\mathbb{P}(\widetilde{E})}(1))^{r}\
= \ \textrm{deg}(\widetilde{E})[pt].
\end{equation}

We aim to establish similar relations for $\pr$.
Let $\mathcal{L}$ be the line bundle on $\pr$ corresponding to the fiber of $\pi$ over
a point in $X\setminus D$, say $x$. Write 
$$\pi^{\ast}(\mathcal{O}_{X}(x))\,=\,\mathcal{L}.$$
By \eqref{0} we have $$(p\,\circ\,\widetilde{\pi})^{\ast}(\mathcal{O}_{X}(x))\,=\, \widetilde{p}^{\,\ast}(\mathcal{L}).$$ Let 
$$p^{-1}(x)\,=\, \{y_{1},\,\cdots,\, y_{|\Gamma|}\},  \qquad y_{i}\,\in\,Y,$$
and for each $i\,\in\,\{1,\,\cdots, \,|\Gamma|\}$, define

$$\widetilde{\mathcal{L}}_{i}\,:=\,  \widetilde{\pi}^{\,\ast}(\mathcal{O}_{Y}(y_{i})).$$
Then
\begin{equation}\label{fiber}
\widetilde{p}^{\,\ast}(\mathcal{L})\,=\, \widetilde{\pi}^{\ast}p^{\ast}(\mathcal{O}_{X}(x))\,=\, \widetilde{\pi}^{\ast}\Big(\bigotimes_{i=1}^{|\Gamma|}\mathcal{O}_{Y}(y_{i})\Big)\,=\,\widetilde{\mathcal{L}}_{1}\,\otimes \widetilde{\mathcal{L}}_{2}\,\otimes\,\cdots\,\otimes\widetilde{\mathcal{L}}_{|\Gamma|}\,:=\,\widetilde{L} . 
\end{equation}
Note that the fiber of $\pr$ at an unramified point of $X$ (i.e., a point 
in $X\setminus D$) is isomorphic to $\mathbb{P}^{r-1}$. Hence, restricting the diagram 
in \eqref{0} to one of the $y_{i}$, we obtain the following commutative diagram
\begin{equation}\label{z1}
\begin{tikzcd} 
\mathbb{P}^{r-1} \arrow[r, hook, "" ] \arrow[d,""]
& \mathbb{P}(\widetilde{E})\arrow[d , swap,""] \\
\mathbb{P}^{r-1}\arrow[r, hook, ""]
& |[, rotate=0]| \pr
\end{tikzcd}
\end{equation}
From \eqref{z1} it can be seen that $\Oe\big\rvert_{\mathbb{P}^{r-1}}\ = \ \mathcal{O}_{\mathbb{P}^{r-1}}
(t)$. 

We also deduce the following:
\begin{equation*}
c_{1}(\mathcal{L})\,\cap\,[\pr]\ = \ [\pi^{-1}(x)]\ = \ [\mathbb{P}^{r-1}]\ \,\in\, \textrm{CH}_{r-1}(\pr).
\end{equation*}

\begin{lemma}\label{numbers}
With the above notation in place, the following relations hold in $\emph{CH}_{\ast}(\pr)$:
\begin{enumerate}
\item [$(i)$] 
$c_{1}(\mathcal{L})^{2}\ = \ 0.$

\item[$(ii)$] $c_{1}(\Oe)^{r-1}\cdot c_{1}(\el)\ = \ {t^{r-1}}[pt]$.

\item[$(iii)$]
$ c_{1}(\Oe)^{r}\ = \ {t^{r}}\emph{par-deg}(E_{\ast})[pt]$.
\end{enumerate}
\end{lemma}

\begin{proof} $(i)$: By \cite[Proposition 3.5]{BBM}, it is evident that 
\begin{equation*}
 \ps(c_{1}(\el))\,=\,\sum^{|\Gamma|}_{i=1}c_{1}(\widetilde{\mathcal{L}}_{i}).
\end{equation*}
Thus,
$|\Gamma|  c_{1}(\mathcal{L})^{2}\, =\, \widetilde{p}_{\ast}\ps(c_{1}({\mathcal{L}})^{2})\,=\, \widetilde{p}_{\ast}\left(\sum^{|\Gamma|}_{i=1}c_{1}(\widetilde{\mathcal{L}}_{i})\right)^{2}\,=\,0$, and 
hence $c_{1}(\mathcal{L})^{2}\ = \ 0.$

$(ii)$: We have the following: $$|\Gamma| c_{1}(\Oe)^{r-1}\cdot c_{1}(\el)\,=\, \widetilde{p}_{\ast}\ps(c_{1}(\Oe)^{r-1}\cdot c_{1}(\el))$$
\begin{equation*}
\hspace{1cm}=\  \frac{1}{|\Gamma|}\widetilde{p}_{\ast}\left(t^{r-1}c_{1}(\mathcal{O}_{\mathbb{P}(\widetilde{E})}(1))^{r-1}\cdot\sum^{|\Gamma|}_{i=1}c_{1}(\widetilde{\mathcal{L}}_{i})\right) 
\end{equation*}
\begin{equation*} 
\hspace{-1cm}=\ \frac{t^{r-1}}{|\Gamma|}\widetilde{p}_{\ast}\left(\sum^{|\Gamma|}_{i=1}[pt_{i}]\right)\,=\, t^{r-1}[pt], 
\end{equation*}
where $[pt_{i}]\,\in\,\textrm{CH}_{0}(\widetilde{\pi}^{-1}(y_{i}))$ for $y_{i}\,\in\,Y$, and 
$[pt]\,\in\,\textrm{CH}_{0}(\pi^{-1}(x))\,=\,\textrm{CH}_{0}(\mathbb{P}^{r-1} )\,=\,\mathbb{Z}.$

$(iii)$:  Using \cite[Proposition 3.5]{BBM} and \eqref{taut} it follows that

$$\widetilde{p}_{\textrm{sp}}^{\,\ast}\Big(c_{1}(\Oe)^{r}\Big)\,=\,t^{r}c_{1}(\mathcal{O}_{\mathbb{P}(\widetilde{E})}(1))^{r}.$$ 
We have
\begin{equation*}
|\Gamma|c_{1}(\Oe)^{r}\,=\, \widetilde{p}_{\ast}\widetilde{p}_{\textrm{sp}}^{\,\ast}\Big(c_{1}(\Oe)^{r}\Big)\,=\, t^{r}\widetilde{p}_{\ast}\Big(c_{1}(\mathcal{O}_{\mathbb{P}(\widetilde{E})}(1))^{r}\Big)\,=\, t^{r}\textrm{deg}(\widetilde{E})[pt], 
\end{equation*}
where $[pt]\,\in\, \textrm{CH}_{0}(\mathbb{P}(E_{\ast}))$. Hence $c_{1}(\Oe)^{r})\,=\, t^{r}\textrm{par-deg}(E_{\ast})[pt]$.
\end{proof}

The diagram in \eqref{0} induces the following commutative diagram of Picard groups
\begin{equation} 
\begin{tikzcd} 
\textrm{Pic}(\mathbb{P}(E_{\ast}))  \arrow[d,"\widetilde{p}^{\,\ast}"]
& \textrm{Pic}(X)\arrow[d , swap,"p^{\ast}"] \arrow[l, "{\pi}^{\ast}"] \\
\textrm{Pic}(\mathbb{P}(\widetilde{E}))
& |[, rotate=0]| \textrm{Pic}(Y) \arrow[l, "\widetilde\pi^{\ast}"].
\end{tikzcd}
\end{equation}

Denote the line bundles $\Oe \, \longrightarrow \, \pr$ and
$\mathcal{O}_{\mathbb{P}(\widetilde{{E}})}(1)\, \longrightarrow \, \mathbb{P}(\widetilde{E})$ by $\xi$
and $\widetilde{\xi}$ respectively. Let $\widetilde{L}$ (respectively, $\widetilde{L}^{'}$) be the line bundle on
$\mathbb{P}(\widetilde{E})$ corresponding to the fiber over $x\,\in\, X\setminus D$
(respectively, $ x^{\prime}\, \in\, D$) under the map $p\circ \widetilde{\pi}$ (see \eqref{0}). For notational convenience, denote the numerical classes as follows:
\begin{equation*}
\widetilde{\el}\, = \, [\widetilde{\pi}^{\ast}\mathcal{O}_{Y}(y)]\, \in\, {N}^{1}(\mathbb{P}(\widetilde{E}))
\,\,\, \textrm{ and }\,\,\, \widetilde{\xi} \ = \ [\mathcal{O}_{\mathbb{P}(\widetilde{E})}(1)]
\, \in\, {N}^{1}(\mathbb{P}(\widetilde{E}))
\end{equation*}
for some $\, y\,\in\, Y$. Similarly, let
$$\xi\, = \, [\Oe]\, \in\, {N}^{1}(\pr)\quad,\quad \el \, = \, [\pi^{\ast}\mathcal{O}_{X}(x)]\, \in \,{N}^{1}(\pr)$$
and
$$\el^{'} \,=\,[\pi^{\ast}\mathcal{O}_{X}(x^{\prime})]\,\in\,{N}^{1}(\mathbb{P}(E_{\ast})).$$

Note that $${p}^{-1}(x) \, = \, \{y_{1},\,\cdots,\,y_{|\Gamma|}\}$$while
$${p}^{-1}(x^{\prime}) \, = \, \{y'_{1},\,
\cdots,\, y'_{m}\},$$ where each $y'_{i}$ occurs with multiplicity $e_{i}$, and $e_{1}+\,\ldots+\,e_{m}
\, = \, |\Gamma|$. Consider
\begin{equation*}
\widetilde{p}^{\,\ast}\, : \, {N}^{1}(\pr)\, \longrightarrow \, {N}^{1}(\mathbb{P}(\widetilde{E}))\quad
\text{ and }\quad{p}^{\ast}\, : \, {N}^{1}(X)\, \longrightarrow \, {N}^{1}(Y).
\end{equation*}
We have $p^{\ast}\Big([\mathcal{O}_{X}(x)]\Big) \ = \ p^{\ast}\Big([\mathcal{O}_{X}(x^{\prime})]\Big)\, = \,
[\mathcal{O}_{Y}(y)^{\otimes |\Gamma|}]$,\, for some $y\,\in\, Y$. Hence,
\begin{equation*}
(p\circ \widetilde{\pi})^{\ast}\Big([\mathcal{O}_{X}(x)]\Big) \, = \, \widetilde{\mathcal{L}}^{\otimes |\Gamma|} \, =
\, (p\circ \widetilde{\pi})^{\ast}\Big([\mathcal{O}_{X}(x^{\prime})]\Big).
\end{equation*}

\begin{remark}\label{equal}
Since any two points on a smooth projective curve are numerically equivalent, the classes $\mathcal{L}$ and $\mathcal{L}^{\prime}$ are numerically equivalent. Moreover,
\[
\widetilde{p}^{\,*}(\mathcal{L})
\,=\,
\widetilde{p}^{\,*}(\mathcal{L}')
\,=\,
\widetilde{\mathcal{L}}^{\otimes |\Gamma|}.
\]
Although the divisor classes associated with ramified and unramified points are numerically equivalent, the corresponding fibers need not be isomorphic. Indeed, if $x\, \in\, X \setminus D$, then the fiber of the projection
\[
\mathbb{P}(E_{\ast})\, \xlongrightarrow{\pi}\, X
\]
is naturally isomorphic to $\mathbb{P}^{r-1}$. In contrast, for a ramified point $x\,\in\, D$, the isotropy subgroup of the covering $p\,:\,Y\, \longrightarrow\, X$ acts nontrivially on this product, and the fiber is obtained as the corresponding quotient. Thus, while the distinction between ramified and unramified points is invisible at the level of numerical equivalence classes, it remains relevant from a geometric viewpoint.
\end{remark}

\subsection{Pseudo-effective cone of a parabolic projective bundle}

To compute the pseudoeffective cone, two separate cases will be considered:\, $E_{\ast}$ is parabolic semistable and $E_{\ast}$ is unstable. We retain the notation
employed in the previous subsection.

Consider the pullback map $\textrm{CH}_{k}(\pr)\ \xrightarrow{\,\,\,\widetilde{p}^{\,\ast}\,\,\,}\
\textrm{CH}_{k}(\mathbb{P}(\widetilde{E}))$.
As proper pushforwards descend from $\textrm{CH}_{k}$ to $\textrm{N}_{k}$ (see \cite[Example 19.1.6]{Fu}), we have the following commutative diagram 
\begin{equation}\label{dia2} 
\begin{tikzcd} 
\textrm{CH}_{k}(\mathbb{P}(\widetilde{E})) \arrow[r, "\widetilde{p}_{\ast}"] \arrow[d, ""] 
& \textrm{CH}_{k}(\pr) \arrow[d, ""] \\
{N}_{k}(\mathbb{P}(\widetilde{E})) \arrow[r, "\widetilde{p}_{\ast}"] 
& {N}_{k}(\pr).
\end{tikzcd}
\end{equation}
By abuse of notation, the above induced map $N_{k}(\mathbb{P}(\widetilde{E}))\,\longrightarrow\,N_{k}(\mathbb{P}
({E_{\ast}}))$ is also denoted by $\widetilde{p}_{\ast}$.

\begin{lemma}\label{pushforward}
Let $\widetilde{p}_{\ast}\,:\, {N}_{r-k}(\mathbb{P}(\widetilde{E})) \, \longrightarrow \,
{N}_{r-k}(\pr)$ be the pushforward map in \eqref{dia2}. Then, for every $k\,\in\,\{1,\,\cdots,\, r\}$, the 
following equalities hold:
\begin{equation*}
\widetilde{p}_{\ast}(c_{1}(\widetilde{\xi})^{k})\, =\,\frac {|\Gamma|}{t^{k}} \,\ c_{1}({\xi})^{k},\ \quad c_{1}(\widetilde{\mathcal{L}})\,=\, c_{1}(\mathcal{L})
\end{equation*}
and
\begin{equation*}
 \widetilde{p}_{\ast}(c_{1}(\widetilde{\xi})^{k-1}\cdot c_{1}(\widetilde{\el}))\, =\,\frac{1}{t^{k-1}}( c_{1}({\xi})^{k-1}\cdot c_{1}(\el)). 
\end{equation*}
\end{lemma}

It should be clarified that although we have expressed the elements of ${N}_{r-k}(\mathbb{P}(\widetilde{E}))$ and ${N}_{r-k}(\pr)$ as homogeneous Chern polynomials of weight $k$, they are actually equivalence classes of Chern polynomials intersected with the cycles $[\mathbb{P}(\widetilde{E})]$ and $[\pr]$ respectively (see Convention \ref{con1}).

\begin{proof}[{Proof of Lemma \ref{pushforward}}] By \cite[Proposition 3.5]{BBM}, 
\begin{equation*}
\ps(c_{1}(\mathcal{L}))\,=\,c_{1}(\widetilde{L})\quad \textrm{and}\quad \ps(c_{1}(\xi))\,=\, t\, c_{1}(\widetilde{\xi}).
\end{equation*}
Using \eqref{comp} it follows that
\begin{equation*}
|\Gamma|\,c_{1}(\xi)^{k}\,=\, \widetilde{p}_{\ast}\ps(c_{1}(\xi)^{k})\,=\,\widetilde{p}_{\ast}(t^{k}c_{1}(\widetilde{\xi})^{k})
\end{equation*}
and similarly,
\begin{equation*}
 |\Gamma|\,c_{1}(\mathcal{L})\,=\,\widetilde{p}_{\ast}\ps(c_{1}(\mathcal{L}))\,=\, \widetilde{p}_{\ast}(c_{1}(\widetilde{L})).
\end{equation*}

Recall that $c_{1}(\widetilde{L})\,=\,|\Gamma|c_{1}(\widetilde{\mathcal{L}})$
as elements of $N_{r-1}(\mathbb{P}(\widetilde{E}))$. Passing these to numerical equivalence via the commutative diagram~\eqref{dia2}, it is deduced that
\begin{equation}\label{1}
|\Gamma|\,c_{1}(\xi)^{k}\,=\,\widetilde{p}_{\ast}(t^{k} c_{1}(\widetilde{\xi})^{k}),
\end{equation}
hence
\begin{equation*}
\widetilde{p}_{\ast}(c_{1}(\widetilde{\xi})^{k})\,=\, \frac{|\Gamma|}{t^{k}}c_{1}(\xi)^{k},   
\end{equation*}
\textrm{and} 
\begin{equation}\label{2}
c_{1}(\mathcal{L})\ \,=\ \,\widetilde{p}_{\ast}(c_{1}(\widetilde{\mathcal{L}}))
\end{equation}
in $N_{r-1}(\mathbb{P}(E_{\ast}))$.
Next, consider
\begin{equation*}
\ps(c_{1}(\xi)^{k-1}\cdot c_{1}(\mathcal{L}))\ =\ \ps(c_{1}(\xi)^{k-1})\cdot \ps(c_{1}(\mathcal{L})).
\end{equation*}
Taking the pushforward of the above cycles and applying \eqref{comp},
\begin{equation*}
|\Gamma|(c_{1}(\xi)^{k-1}\cdot c_{1}(\mathcal{L}))\,=\, \widetilde{p}_{\ast} \ps(c_{1}(\xi)^{k-1} \cdot c_{1}(\mathcal{L}))
\end{equation*}
\begin{equation*}
\hspace{11em}=\, \widetilde{p}_{\ast}(t^{k-1}c_{1}(\widetilde{\xi})^{k-1} \cdot c_{1}(\widetilde{L})). 
\end{equation*}
Now using \eqref{1} and \eqref{2}, we have
\begin{equation*}
\widetilde{p}_{\ast}(c_{1}(\widetilde{\xi})^{k-1}\cdot c_{1}(\widetilde{\el}))\ \,=\ \,\frac{1}{t^{k-1}}(c_{1}(\xi)^{k-1}\cdot c_{1}(\mathcal{L}))
\end{equation*}
in $N_{k}(\pe)$.
\end{proof}

The following theorem determines the generators of the pseudoeffective cone of $\mathbb{P}(E_{\ast})$ for a parabolic vector bundle $E_{\ast}$.

\begin{theorem}\label{main thm}
Let $E_{\ast}$ be a parabolic vector bundle of rank $r$ and parabolic degree $d$ on a curve $X$. Let
$$
E\, =\, E^{0}\,\supsetneq\, E^{1}\, \supsetneq\, \cdots\,\supsetneq\, E^{l}\, \supsetneq\, E^{l+1}\, = \,0
$$
be the Harder-Narasimhan filtration of $E_{\ast}$ (as in \eqref{HN}). As in \eqref{eqs1} and \eqref{parrank},
define
\begin{equation*}
r_{i} \, := \, \emph{rank}\,Q^{{i}}_{\ast}, \quad d_{i}\, := \, \emph{par-deg}\,Q^{i}_{\ast}, \quad \underline{r}_{i}\, := \, \emph{rank}(E/E^{{i}})_{\ast}\, = \, \sum_{t =1}^{i}r_{t},
\end{equation*}
$$
\underline{d}_{i}\, := \, \emph{par-deg}\, E^{{i}}_\ast \, =\, d-\sum_{j=1}^{i}d_{j}.
$$
Define for all $s\,\in\, \{1,\,\cdots,\,l\}$ and $j\,\in\, \{1,\,\cdots,\,r_{s}\}$, except when $s\, = \, l$
and $j\, = \, r_{l}$,
\begin{equation}\label{us}
\upsilon_{\underline{r}_{s-1}+j} \, := \, (j\mu_{s}-\underline{d}_{s-1})t.
\end{equation}
Then, for all $k\,\in\,\{1,\,\cdots,\,r-1\}$, the pseudoeffective cone is given by
\begin{equation*}
\overline{\emph{Eff}}_{k}(\pr)\, = \, \Big\langle c_{1}({\xi})^{r-k}+ \upsilon_{k}\,c_{1}({\xi})^{r-k-1}  \cdot c_{1}({\el}),\quad c_{1}(\xi)^{r-k-1} \cdot c_{1}(\el)\Big\rangle.
\end{equation*}
Furthermore, the pushforward map $\widetilde{p}_{\ast}\,:\,\overline{\emph{Eff}}_{k}(\mathbb{P}(\widetilde{E}))\,\longrightarrow\, \overline{\emph{Eff}}_{k}(\mathbb{P}({E}_{\ast}))$ is an isomorphism.
\end{theorem}

\begin{proof}
By \cite{Ful}, the pseudoeffective cone of $\mathbb{P}(\widetilde{E})$, for all $k\,\in\,\{1,\,\cdots,\, r-1\}$, 
\begin{equation*}
\overline{\textrm{Eff}}_{k}(\mathbb{P}(\widetilde{E}))\ =\, \ \left\langle c_1(\widetilde{\xi})^{r-k} \,+\, \nu_{k}\, c_1(\widetilde{\xi})^{r-k-1}\cdot c_{1}(\widetilde{\mathcal{L}}),\quad c_{1}(\widetilde{{\xi}})^{r-k-1}\cdot c_{1}(\widetilde{{\el}})\right\rangle,
\end{equation*}
where $\nu_{k}$ is the number defined in \cite[Theorem 1.1]{Ful}. By Lemma \ref{pushforward}, for the map ${N}_{k}(\mathbb{P}(\widetilde{E}))\ \xrightarrow{\,\,\,\widetilde{p}_{\ast}\,\,\,}\
{N}_{k}(\pr)$, we obtain the following:
\begin{equation}\label{gen1}
\widetilde{p}_{\ast}\Big(c_{1}(\widetilde{\xi})^{r-k} \,+\,\nu_{k}\,c_{1}(\widetilde{\xi})^{r-k-1}\cdot c_{1}(\widetilde{\el}))\Big)
\,=\, \frac{|\Gamma|}{t^{r-k}}\, \Big(c_{1}(\xi)^{r-k}\,+\,\frac{{\nu_{k}}\, t}{|\Gamma|}\, c_{1}(\xi)^{r-k-1}\cdot c_{1}(\el) \Big),
\end{equation}
and 
\begin{equation}\label{gen2}
\widetilde{p}_{\ast}\Big(c_{1}(\widetilde{{\xi}})^{r-k-1}\cdot c_{1}(\widetilde{{\el}})\Big)\,  =
\, \frac{1}{t^{r-k-1}} \Big(c_{1}({\xi})^{r-k-1}\cdot c_{1}({\el})\Big).
\end{equation}
Using induction on the length of the Harder--Narasimhan filtration of $E_*$ it follows that for all $k\, \in\, \{1, \,\cdots,\, r\,-\,1\}$, the number defined in \eqref{us} satisfies the condition
\[
\upsilon_k\, =\, \frac{\nu_k}{|\Gamma|}\, t.
\]
We will show that 
\begin{equation}\label{2g}
c_{1}({\xi})^{r-k-1}\cdot c_{1}({\el})
\quad \textrm{and}\quad c_{1}(\xi)^{r-k}\,+\, \upsilon_{k}\,c_{1}(\xi)^{r-k-1} \cdot c_{1}(\el)
\end{equation}
generate the pseudoeffective cone $\overline{\textrm{Eff}}_{k}(\pr)$. First, note that by \eqref{gen1}, \eqref{gen2}  and \cite[Corollary 3.8]{FL1}, neither of the two classes is numerically trivial.

Next, we will show that the two classes in \eqref{2g} are not numerically equivalent. For this, consider the weight 
$k$--Chern polynomial $c_{1}(\xi)^{k-1}\cdot c_{1}({\mathcal{L}})\,\in\,\textrm{CH}^{k}(\pr)$. Then by Lemma \ref{numbers} we have 
\begin{equation*}
\Big(c_{1}({\xi})^{r-k}\,+\, \upsilon_{k}\,c_{1}({\xi})^{r-k-1}  \cdot c_{1}({\el})\Big)\,\cap\,\Big ( c_{1}(\xi)^{k-1}\cdot c_{1}(\mathcal{L})\Big)
\end{equation*}
$$
=\ c_{1}(\xi)^{r-1}\cdot c_{1}(\mathcal{L})
\ =\ t^{r-1}\ \in\ \textrm{CH}_{0}(\pr),
$$
and
\begin{equation*}
\Big( c_{1}(\xi)^{r-k-1}\cdot c_{1}(\mathcal{L})\Big) \ \cap\  \Big( c_{1}(\xi)^{k-1}\cdot c_{1}(\mathcal{L})\Big)\ =\ 0 \ \in\ \textrm{CH}_{0}(\pr).
\end{equation*}
Thus, the two classes in \eqref{2g} are not numerically equivalent and are linearly independent. Now by \cite[Corollary 3.22]{FL2} it follows that
\begin{equation*}
\overline{\textrm{Eff}}_{k}(\pr)\ =\ \Big\langle c_{1}({\xi})^{r-k}\,+\, \upsilon_{k}\,c_{1}({\xi})^{r-k-1}  \cdot c_{1}({\el}),\quad c_{1}(\xi)^{r-k-1} \cdot c_{1}(\el)\Big\rangle
\end{equation*}
for all $k\,\in\,\{1,\,\cdots,\,r-1\}$. Finally, since the generators of $\overline{\textrm{Eff}}_{k}(\mathbb{P}(\widetilde{E}))$ map bijectively onto the generators of $\overline{\textrm{Eff}}_{k}(\pe)$, and both cones have the same number of linear independent generators, it follows that $\widetilde{p}_{\ast}$ is an isomorphism.
\end{proof}

\begin{corollary}
For a parabolic vector bundle $E_{\ast}$ of rank $r$ on $X$, and for any $k\,\in\,\{1,\,\cdots,\,r-1\}$, the numerical group ${N}_{k}(\pr)$ is two-dimensional and is generated by the classes
\begin{equation*}
c_{1}({\xi})^{r-k} \quad\text{and}\quad c_{1}(\xi)^{r-k-1} \cdot c_{1}(\el).
\end{equation*}
\end{corollary}

\begin{proof}
This follows immediately from Theorem~\ref{main thm}, because the pseudoeffective cone $\overline{\textrm{Eff}}_{k}(\pr)$ is a full-dimensional subcone of ${N}_{k}(\pr)$. 
\end{proof}

Now $\textrm{Nef}^{\,k}(\pr)$ and $\overline{\textrm{Eff}}^{\,k}(\pr)$
will be computed for all $k\,\in\, \{1,\,\cdots,\, r-1\}$ in $N^{k}(\pr)$. Recall that elements of $N^{k}(\pr)$ are numerical equivalence classes of homogeneous Chern polynomials of weight
$k$ (see \eqref{upper numerical groups}).

\begin{corollary}\label{higher nef}
For a parabolic vector bundle $E_{\ast}$ of rank $r$ on $X$, the  nef cone $\emph{Nef}^{\,k}(\pr)$ admits the following description for every $k\,\in\,\{1,\,\cdots,\,r-1\}$, 
\begin{align*}
\emph{Nef}^{\,k}(\pr)\, & = \, \left\langle \frac{1}{t^{r-1}}c_{1}(\xi)^{k} - \frac{td\,+\,\nu_{k}}{t^{r-1}}c_{1}(\xi)^{k-1} \cdot  c_{1}(\el),\quad \frac{1}{t^{r-1}}c_{1}({\xi})^{k-1} \cdot  c_{1}(\el)\right\rangle \\\\
& = \,\Big\langle c_{1}(\xi)^{k} \,-\, (td\,+\,\nu_{k})c_{1}(\xi)^{k-1} \cdot  c_{1}(\el),\quad c_{1}({\xi})^{k-1} \cdot  c_{1}(\el)\Big\rangle .
\end{align*}
\end{corollary}

\begin{proof}
Recall the intersection pairing 
$$N^{k}(\pr)\, \times \, N_{k}(\pr)\, \longrightarrow \, \mathbb{R},\qquad (P, \alpha)\, \longmapsto \ P\,\cap\, \alpha,$$ defined in \eqref{ch}. By definition, the cone  $\textrm{Nef}^{\,k}(\pr)$ is the
dual cone of $\overline{\textrm{Eff}}_{k}(\pr)$ with respect to this pairing. Therefore, the description of $\textrm{Nef}^{\,k}(\pr)$ follows immediately from the description of 
$\overline{\textrm{Eff}}_{k}(\pr)$ given in Theorem \ref{main thm}.
\end{proof}

\begin{corollary}\label{higher eff}
Let $E_{\ast}$ be a parabolic vector bundle of rank $r$ on $X$. Then, for every $k\,\in\,\{1,\,\cdots,\,r-1\}$, the dual pseudoeffective cone $\overline{\emph{Eff}}^{\,k}(\pr)$ is given by
\begin{equation*}
\overline{\emph{Eff}}^{\,k}(\pr)\ = \
\Big\langle c_{1}({\xi})^{k}\,+\, \nu_{r-k}\,c_{1}({\xi})^{k-1} \cdot   c_{1}({\el}), \quad c_{1}(\xi)^{k-1} \cdot c_{1}(\el)\Big\rangle.
\end{equation*}
\end{corollary}

\begin{proof}
Recall the map
\[
\phi \ \colon\ N^{k}(\pr)\ \longrightarrow\ N_{r-k}(\pr),
\qquad P\ \longmapsto\ P\,\cap\,[\pr],
\]
defined in \eqref{isomorphism}. By Definition \ref{dp}, we have
\[
\overline{\mathrm{Eff}}^{\,k}(\pr)
\ =\ \phi^{-1}\,\!(
\overline{\mathrm{Eff}}_{\,r-k}(\pr).
\]
Therefore, the result follows immediately from Theorem~\ref{main thm}.
\end{proof}

The following corollary specializes Theorem~\ref{main thm} to the parabolic semistable case and yields a particularly simple description of the pseudoeffective cones.

\begin{corollary}\label{eff semi-stable}
Let $E_{\ast}$ be a parabolic semistable vector bundle of rank $r$ and
parabolic slope $\mu$. Then for all $k\,\in\, \{1,\,\cdots,\,r-1\}$,
\label{eff semistable}
\begin{equation*}
\overline{\emph{Eff}}_{k}(\pr)\ =\, \Big\langle
(c_{1}(\xi)\,-\, n c_{1}(\el))^{r-k},
\quad
c_{1}(\xi)^{r-k-1}\cdot c_{1}(\el)
\Big\rangle ,
\end{equation*}
where $n \, =\,\mu t$. 
\end{corollary}

\begin{proof}
Since $E_{\ast}$ is parabolic semistable, we have $l\, =\, 1$ in the Harder--Narasimhan filtration of $E_{\ast}$ (see \eqref{HN}). This is a special case of Theorem \ref{main thm} with $\upsilon_k \,=\, (r\, -\, k)\mu t$.
\end{proof}

Now assume that $E_{\ast}$ is parabolic unstable (meaning not parabolic semistable). We shall show that the following isomorphisms hold, exactly as in the case of ordinary projective bundles:
\begin{equation*}
\overline{\mathrm{Eff}}_{k}\big(\mathbb{P}(\widetilde{Q}^{1})\big) \ \cong\ \overline{\mathrm{Eff}}_{k}\big(\mathbb{P}(\widetilde{E})\big)
\quad \text{and} \quad
\overline{\mathrm{Eff}}_{k}\big(\mathbb{P}(\widetilde{E}^{1})\big) \ \cong\ \overline{\mathrm{Eff}}_{k}\big(\mathbb{P}(E^{1}_{\ast})\big)
\end{equation*}
for $1 \,\le\, k \,\le\, r_{1}-1$ and $r_{1}+1 \,\le\, k \,\le \,r-1$, respectively, where $r_{1} \,=\, \mathrm{rk}(\widetilde{Q}^{1})$ (see \eqref{eqs1} for the definition of $\widetilde{Q}^{1}$); see \cite[Proposition~1.3 and Proposition~1.4]{Ful}. 

Since $E_{\ast}$ is not parabolic semistable, the orbifold bundle $\widetilde{E}$ on $Y$ corresponding to $E_{\ast}$ is also unstable (see Proposition \ref{semistable}). Let 
\begin{equation*}
0\ \longrightarrow\ E^{1}_{\ast}\ \longrightarrow\ E_{\ast}\ \longrightarrow\ Q^{1}_{\ast}\ \longrightarrow\ 0
\end{equation*}
be the short exact sequence induced by the Harder-Narasimhan filtration of $E_{\ast}$. 
By the correspondence between parabolic and orbifold bundles, we obtain a corresponding short exact sequence of orbifold bundles $Y$ 
\begin{equation*}
0\ \longrightarrow\ \widetilde{E}^{1}\ \longrightarrow\ \widetilde{E}\ \longrightarrow\ \widetilde{Q}^{1}\ \longrightarrow\ 0,
\end{equation*}
where $E^{1}_{\ast}$ and $Q^{1}_{\ast}$ correspond to $\widetilde{E}^{1}$ and $\widetilde{Q}^{1}$ respectively. 

The quotient map $\widetilde{E}\,\twoheadrightarrow\, \widetilde{Q}^{1}$ induces a natural closed embedding 
\begin{equation*}
\widetilde{i}\ : \ \mathbb{P}(\widetilde{Q}^{1}) \ \longrightarrow \ \mathbb{P}(\widetilde{E}).
\end{equation*}
Furthermore, the restriction of the quotient morphism 
$$\widetilde{p}\ :\ \mathbb{P}(\widetilde{E})
 \ \longrightarrow\ \mathbb{P}(E_{\ast})$$
 to $\mathbb{P}(\widetilde{Q}^{1})$ induces an isomorphism 
$$\widetilde{p}(\mathbb{P}(\widetilde{Q}^{1}))\ \cong\ \mathbb{P}(\widetilde{Q}^{1})/\Gamma\ =\ 
\mathbb{P}(Q^{1}_{\ast}).$$ 
Consequently, there is a natural closed embedding $i\, : \, \mathbb{P}(Q^{1}_{\ast}) \, \longrightarrow \,\mathbb{P}(E_{\ast})$.

\begin{corollary}\label{lower eff}
Let $E_{\ast}$ be a parabolic unstable vector bundle of rank $r$ and
parabolic slope $\mu$. Retaining the notation introduced above, let $n\,:=\, \mu_{1}t$. Then, for every $k\,\in\, \{1,\,\cdots,\,r_{1}\}$, the pseudoeffective cone is generated by the classes
\begin{equation*}
\overline{\emph{Eff}}_{k}(\mathbb{P}(E_{\ast})) \, = \,\left\langle [\mathbb{P}(Q^{1}_{\ast})] \cdot (c_{1}(\xi)\,-\,n c_{1}(\el))^{r_{1}-k},\quad c_{1}(\xi)^{r-k-1}\cdot c_{1}(\el)\right\rangle.
\end{equation*}
Here, the remaining notations are as in $\eqref{parrank}$. Furthermore, $i_{\ast}$ induces an isomorphism
$\overline{\emph{Eff}}_{k}(\mathbb{P}(Q^{1}_{\ast}))\, \cong \, \overline{\emph{Eff}}_{k}(\mathbb{P}(E_{\ast}))$ for all
$1\, \leq\, k\,\leq \, r_{1}-1$.
\end{corollary}

\begin{proof}
We set 
$$\widetilde{\mu}_{1}\, := \, \mu(\widetilde{Q}^{1}).$$
Observe that
\begin{equation*}
r_1 \ = \ \rank(\widetilde{Q}^1)\, = \, \rank(Q^1_{\ast}),\quad \textrm{deg}(\widetilde{Q}^{1})\, = \, |\Gamma|d_{1}\, =: \, \widetilde{d_{1}}\quad \textrm{and} \quad\textrm{deg}(\widetilde{E})\, = \, |\Gamma|d\, =: \, \widetilde{d}.
\end{equation*}
Invoking \cite[Exercise 3.2.17]{Fu} --- adjusted to bundles of quotients over curves --- it is deduced  that 
\[
\widetilde{\delta}\ :=\ [\mathbb{P}(\widetilde{Q}^{1})]\ =\ c_{1}(\widetilde{\xi})^{r-r_{1}} \,+\, (\widetilde{d_{1}}\,-\,\widetilde{d})c_{1}(\widetilde{\xi})^{r-r_{1}-1}\cdot c_{1}(\widetilde{\el})
\]
\[
= \ c_{1}(\widetilde{\xi})^{r-r_{1}} \,+\, |\Gamma|(d_{1}\,-\,d)c_{1}(\widetilde{\xi})^{r-r_{1}-1}\cdot c_{1}(\widetilde{\el}).
\]
Define
\begin{equation*}
\delta \, := \, c_{1}(\xi)^{r-r_{1}}\,+\,((d_{1}\,-\,d) t) c_{1}(\xi)^{r\,-\,r_{1}-1} \cdot c_{1}(\el),
\end{equation*}
Apply Lemma \ref{pushforward}, we obtain the following:
\begin{equation*}
\widetilde{p}_{\ast}\left(\widetilde{\delta}\cdot\left(c_{1}(\widetilde{\xi})\,-\,\widetilde{\mu}_{1}c_{1}(\widetilde{\el})\right)^{r_{1}-k}\right)\, = \, \frac{|\Gamma|}{t^{r-k}} \, (\delta\cdot\left(c_{1}(\xi) \,-\,n c_{1}(\el)\right)^{r_{1}-k}).
\end{equation*}
Similarly,
\begin{equation*}
\widetilde{p}_{\ast}\left(c_{1}(\widetilde{\xi})^{r-k-1} \cdot c_{1}(\widetilde{\el})\right)\, \vspace{-1mm}\,= \, \frac{1}{t^{r-k-1}} \, c_{1}(\xi)^{r-k-1} \cdot c_{1}(\el).
\end{equation*}
By \cite[Lemma 2.3]{Ful}, for every $k \,\in\, \{1,\,\cdots,\, r_{1}\}$, the pseudoeffective cone is given by
\begin{equation}\label{eff}\
\overline{\textrm{Eff}}_{k}(\mathbb{P}(\widetilde{E}))\, =\, \left\langle [\mathbb{P}(\widetilde{Q}^{1})]\cdot\left(c_{1}(\widetilde{\xi})-\widetilde{\mu}_{1}c_{1}(\widetilde{\el})\right)^{r_{1}-k},\quad c_{1}(\widetilde{\xi})^{r-k-1} c_{1}(\widetilde{\el})\right\rangle.
\end{equation}
It follows from Lemma~\ref{pushforward} that
\[
\widetilde p_*(\widetilde \delta)
\,=\,
\frac{|\Gamma|}
{t^{r-r_{1}}}
\,\delta.
\]
On the other hand, by definition, we have
\[
\widetilde p_*
\Big(
\big[
\mathbb P(\widetilde Q^{1})
\big]
\Big)
\,=\,
|\Gamma|\,
\big[
\mathbb P(Q_{*}^{1})
\big].
\]
Since \(\widetilde \delta\,=\,
\big[
\mathbb P(\widetilde Q^{1})\big]\), we obtain
\[
\big[
\mathbb P(Q_{*}^{1})
\big]
\,=\, \frac{\delta}{t^{r-r_{1}}}.
\]
Arguing exactly as in the proof of  Theorem \ref{main thm}, we deduce that
\begin{equation*}
\overline{\textrm{Eff}}_{k}(\mathbb{P}(E_{\ast}))\, =\, \left\langle [\mathbb{P}(Q^{1}_{\ast})]\cdot (c_{1}(\xi)\,-\,n c_{1}(\el))^{r_{1}-k},\quad c_{1}(\xi)^{r-k-1}c_{1}(\el)\right\rangle, \,\,\textrm{for all}\,\, k\,\in\, \{1,\,\cdots,\,r_{1}\}.
\end{equation*}
Furthermore, \cite[Lemma 2.3]{Ful}, implies that 
$$\widetilde{i}_{\ast}\,:\,\overline{\textrm{Eff}}_{k}(\mathbb{P}(\widetilde{Q}^{1}))\, \longrightarrow \, \overline{\textrm{Eff}}_{k}(\mathbb{P}(\widetilde{E}))$$
is an isomorphism for every $k\,\in\,\{1,\,\cdots,\, r_{1}-1\}$.

On the other hand, Theorem~\ref{main thm}, shows that the pushforward map $\widetilde{p}_{\ast}$ induces the following isomorphisms:
\begin{equation*}
\overline{\textrm{Eff}}_{k}(\mathbb{P}(\widetilde{E}))\, \cong \, \overline{\textrm{Eff}}_{k}(\mathbb{P}(E_{\ast}))\,\,\,\textrm{and}\,\,\,\overline{\textrm{Eff}}_{k}(\mathbb{P}(\widetilde{Q}^{1}))\, \cong \, \overline{\textrm{Eff}}_{k}(\mathbb{P}(Q^{1}_{\ast})). 
\end{equation*}
Combining these isomorphisms, we conclude that $i_{\ast}\,:\,\overline{\textrm{Eff}}_{k}(\mathbb{P}(Q^{1}_{\ast})) \, \longrightarrow\, \overline{\textrm{Eff}}_{k}(\pe),$ is an isomorphism for all
$k\,\in\,\{1,\,\cdots,\,r_{1}-1\}$. This completes the proof.
\end{proof}

\begin{corollary}\label{Eff E_{1}}
Let $E_{\ast}$ be a parabolic unstable vector bundle of rank $r$ and
parabolic slope $\mu$. For $0 \,<\, k\, \leq\, r - r_{1}-1$,
\begin{equation*}
\overline{\emph{Eff}}_{r_{1}+k}(\mathbb{P}(E_{\ast}))\, \cong\, \overline{\emph{Eff}}_{k}(\mathbb{P}({E}^{1}_{\ast})).
\end{equation*}
\end{corollary}

\begin{proof}
For $0\, <\, k\, \leq\, r \,-\, r_{1}\,-\,1 $, \cite[Lemma 2.7]{Ful} shows that  $$\overline{\textrm{Eff}}_{k}(\mathbb{P}(\widetilde{E}^{1}))\, \cong \, \overline{\textrm{Eff}}_{r_{1}+k}(\mathbb{P}(\widetilde{E})).$$ By abuse of notation, we denote the  quotient map $\mathbb{P}(\widetilde{E}^{1})\,\longrightarrow\, \mathbb{P}(E^{1}_{\ast})$ by $\widetilde{p}$. By Theorem \ref{main thm}, we define an isomorphism $\Phi$ to make the following diagram commute:
\begin{equation*} 
\begin{tikzcd} 
\overline{\textrm{Eff}}_{k}(\mathbb{P}({E}^{1}_{\ast})) \arrow[r, "\Phi"] 
& \overline{\textrm{Eff}}_{r_{1}+k}(\mathbb{P}(E_{\ast})) \\
\overline{\textrm{Eff}}_{k}(\mathbb{P}(\widetilde{E}^{1})) \arrow[u, "\widetilde{p}_{\ast}" swap, "\cong"] \arrow[r, "f_{k}" swap, "\cong"] 
& \overline{\textrm{Eff}}_{r_{1}+k}(\mathbb{P}(\widetilde{E})) \arrow[u, "\widetilde{p}_{\ast}" swap, "\cong"]
\end{tikzcd}.
\end{equation*} 
Here $f_{k}$ denotes the map defined in \cite[Lemma~2.7]{Ful}, and we set
\[
\Phi \,:=\, \widetilde{p}_{*}\, \circ\, f_{k} \,\circ\, \widetilde{p}_{*}^{\,-1}.
\]
Since $\Phi$ is a composition of isomorphisms, it is itself an isomorphism.
\end{proof}

\subsubsection{An application}

We conclude this section with an application of the preceding results to the semistability of parabolic vector bundles. The classical case is discussed in \cite{Ful}.

\begin{definition}
We say a parabolic vector bundle $E_{\ast}$ is parabolic $k$--homogeneous if every pseudoeffective $k$--dimensional cycle on $\pr$ is nef, i.e., $\overline{\textrm{Eff}}^{\,k}(\mathbb{P}(E_{\ast}))\,=\, \textrm{Nef}^{\,k}(\pr)$. 
\end{definition}

\begin{theorem}
A parabolic vector bundle $E_{\ast}$ of rank $r$ on $X$ is parabolic semistable if and only if $E_{\ast}$ is 
$k$--homogeneous for all $k\,\in\, \{1,\,\cdots,\,r-1\}$.
\end{theorem}

\begin{proof}
We begin by making the simple observation that $\overline{\textrm{Eff}}^{\,k}(\mathbb{P}(E_{\ast}))\,=\, \textrm{Nef}^{\,k}(\pr)$ if and only if
\begin{equation*}
-(td\, +\, \nu_k) \,=\, \nu_{r-k}.
\end{equation*}
This follows from the description of generators for the respective cones as stated in Corollary \ref{higher eff} and Corollary \ref{higher nef}. If $E_{\ast}$ is parabolic semistable, then by Theorem \ref{eff semistable}, we have 
\begin{equation*}
\nu_k\, =\, -(r-k)\mu t \quad \textrm{and}\quad \nu_{r-k}\, =\, -k \mu t 
\end{equation*}
for all $1\, \leq\, k\, \leq\, r\,-\,1$. Thus $\nu_k \,+\, \nu_{r-k} \,=\, -dt$ and we conclude that $E_{\ast}$ is parabolic $k$-homogeneous for all $k\,\in\, \{1,\,\cdots,\,r-1\}$.

Conversely, if $E_{\ast}$ is $k$--homogeneous for all $k\,\in\, \{1,\,\cdots\,,r-1\}$, we have $\nu_k\, +\, \nu_{r-k}\, =\, -dt$. In particular, $\nu_1\, +\, \nu_{r-1} \,=\, -dt$. A straightforward computation using Theorem \ref{main thm} yields, $\mu_1\, =\,\mu_l$. This is an impossibility unless $E_{\ast}$ is semistable.
\end{proof}

\section{Positive cone of product of parabolic vector bundles}

Take $X$, $Y$, $D$, $\Gamma$ and $p$ as in Section~3. Let $E_{1*}$ and $E_{2*}$ be parabolic vector bundles of rank $r_{1}$ and $r_{2}$ respectively, on $X$, with each equipped with a parabolic structure along
$D$. Denote by $E_{1}$ (respectively, $E_{2}$) the vector bundle underlying $E_{1*}$ (respectively,
$E_{2*}$); their parabolic weights are rational numbers. Denote by $\widetilde{E}_{1}$ (respectively,
$\widetilde{E}_{2}$) the orbifold bundle on $Y$ corresponding to $E_{1*}$ (respectively, $E_{2*}$).
We have the corresponding parabolic projective bundles
\[
\pi_{1}\,\colon\, \mathbb P(E_{1*}) \,\longrightarrow\, X,\qquad 
\pi_{2}\,\colon\, \mathbb P(E_{2*})\,\longrightarrow\, X .
\]
On the orbifold side, we have the $\Gamma$--equivariant projective bundles
\[
\widetilde{\pi}_{1}\,\colon\, \mathbb{P}(\widetilde{E}_{1})\,\longrightarrow\, Y,
\qquad
\widetilde{\pi}_{2}\,\colon\, \mathbb{P}(\widetilde{E}_{2})\,\longrightarrow\, Y.
\]
Consider the diagonal action of $\Gamma$ on the fiber product
\begin{equation*}
\w{S}\,=\,\mathbb P(\widetilde{E}_{1})\,\times_Y\, \mathbb P(\widetilde{E}_{2}).
\end{equation*}
Now taking the $\Gamma$--quotient of $\w{S}$, we have
\begin{equation*}
S\,:=\,\w{S}/\Gamma\,=\,  \Big(\mathbb P(\widetilde{E}_{1})\,\times_Y\, \mathbb P(\widetilde{E}_{2})\Big)/\Gamma\,=\,\mathbb P(\widetilde{E}_{1})/\Gamma \,\times_{X}\, \mathbb P(\widetilde{E}_{2})/\Gamma 
\end{equation*}
\begin{equation*}
\hspace{6.3cm}=\, \mathbb P(E_{1*})\,\times_X\, \mathbb P(E_{2*}).    
\end{equation*}
Let 
\begin{equation}\label{prod}
\widetilde p\,\ \colon\,\ \widetilde S\,\ \longrightarrow\,\ S
\end{equation}
denotes the above quotient morphism. Furthermore, let
\[
\widetilde{\alpha}_1 \,\colon\, \widetilde{S}\,\longrightarrow\,\mathbb P(\widetilde{E}_{1}),
\qquad
\widetilde{\alpha}_2 \,\colon\, \widetilde{S}\,\longrightarrow\,\mathbb{P}(\widetilde{E}_{2})
\]
and
\[
\alpha_1 \,\colon\, S\,\longrightarrow\, \mathbb{P} (E_{1*}),
\qquad
\alpha_2 \,\colon\, S\,\longrightarrow\,  \mathbb{P} (E_{2*})
\]
denote the natural projection morphisms. The above situation is therefore summarized by the following commutative diagram of morphisms:
\begin{center}
\begin{equation}\label{dia1}
\begin{tikzcd}
Y \arrow[d, "p"]    &  \mathbb{P}(\widetilde{E}_{2}) \arrow[l, "\widetilde{\pi}_{2}"] \arrow[d, "\widetilde{p}_{2}"]    & \widetilde{S} \,:=\, \tef \arrow[l, "\widetilde{\alpha}_{2}"] \arrow[d, "\widetilde{p}"] \arrow[r,"\widetilde{\alpha}_{1}"]   &\mathbb{P}(\widetilde{E}_{1})\arrow[r, "\widetilde{\pi}_{1}"]\arrow[d, "\widetilde{p}_{1}"]    & Y\arrow[d, "p"]\\
X                        &  \pf  \arrow[l, "\pi_{2}"]              & S \,:=\, \ef\arrow[r, "\alpha_{1}"]  \arrow[l, "\alpha_{2}"]           & {\mathbb P}(E_{1*})\arrow[r, "\pi_{1}"]           & X
\end{tikzcd}
\end{equation}
\end{center}

\subsection{Tautological line bundle}

Let $N(E_{1*})$ and $N(E_{2*})$ be the least common multiples of the
orders of the isotropy subgroups of the ramified principal bundles
$E_{1,\mathrm{GL}(r_{1},\mathbb{C})}$ and
$E_{2,\mathrm{GL}(r_{2},\mathbb{C})}$ corresponding to the parabolic vector bundles
$E_{1*}$ and $E_{2*}$ respectively. For simplicity of notation, we write
\[
t_1\,:=\,N(E_{1*}),
\qquad
t_2\,:=\,N(E_{2*}).
\]
By the construction of the tautological line bundles on
$\mathbb{P}(E_{1\ast})$ and $\mathbb{P}(E_{2\ast})$ respectively, we have
\begin{equation}\label{pull1}
\widetilde p^{\,*}\!\left(\alpha_1^*\mathcal O_{\mathbb{P}(E_{1\ast})}(1)\right)
\,=\,
\widetilde\alpha_1^*\mathcal O_{\mathbb P(\widetilde E_1)}(t_1)
\end{equation}
and
\begin{equation}\label{pull2}
\widetilde p^{\,*}\!\left(\alpha_2^*\mathcal O_{\pf}(1)\right)
\,=\,
\widetilde\alpha_2^*\mathcal O_{\mathbb P(\widetilde E_2)}(t_2).
\end{equation}

\subsection{Divisors pulled back from the base curve}

Let $x\,\in\, X\,\setminus\,D$ be an unramified point. Define the following line bundles on $S$ by
\begin{equation}\label{l}
\mathcal L
\,:=\,
(\pi_1\,\circ\,\alpha_1)^*\mathcal O_X(x)\,=\,(\pi_2\,\circ\,\alpha_2)^*\mathcal O_X(x)
\end{equation}
(see \eqref{dia1} for the maps). Similarly, for a point $x^{\prime}\,\in\,D$, define
\[
\mathcal L'
\,:=\,
(\pi_1\,\circ\,\alpha_1)^*\mathcal O_X(x')\,=\,
(\pi_2\,\circ\,\alpha_2)^*\mathcal O_X(x').
\]
We now describe their pullbacks to $\widetilde{S}$. Let
\begin{equation*}
 p^{-1}(x)\ =\ \{y_{1},\,\cdots,\, y_{|\Gamma|}\}, \quad y_{i}\ \in\ Y.   
\end{equation*}
Then 
\begin{equation*}
\widetilde{L}\,:=\, \widetilde{p}^{\,\ast}\mathcal{L}\,=\, (\widetilde{\pi}_{1}\,\circ\,\widetilde{\alpha}_{1})^{\ast} p^{\ast} \mathcal{O}_{X}(x) \,=\, \bigotimes_{i=1}^{|\Gamma|} (\widetilde{\pi}_{1}\,\circ\,\widetilde{\alpha}_{1})^{\ast} (\mathcal{O}_{Y}(y_{i}))\,=\, (\widetilde{\pi}_{2}\circ\widetilde{\alpha}_{2})^{\ast}p^{\ast}\mathcal{O}_{X}(x),   
\end{equation*}
and let 
$$\widetilde{\mathcal{L}}\,:=\, (\widetilde{\pi}_1\,\circ\,\widetilde{\alpha}_1)^*\mathcal O_Y(y)\,=\,(\widetilde{\pi}_2\,\circ\,\widetilde{\alpha}_2)^*\mathcal O_Y(y)$$
for some $y\,\in\,Y$.

Now let $x^{\prime}\,\in\,D$, and write 
\begin{equation*}
p^{-1}(x^{\prime})\,=\, \{y_{1}^{\prime},\,\cdots,\,y_{m}^{\prime}\}, \qquad y_{i}^{\prime}\,\in\, Y.    
\end{equation*}
Assume that $y_{i}^{\prime}$ occurs with multiplicity $e_{i}$ for each $i\,=\,1,\,\cdots,\,m$, so that 
\begin{equation*}
e_{1}\,+\,\cdots\,+\,\,e_{m}\,=\, |\Gamma|.    
\end{equation*}
Then 
\begin{equation*}
\widetilde{L}^{\prime}\,:=\, \widetilde{p}^{\,\ast}\mathcal{L}^{\prime}\,=\, (\widetilde{\pi}_{1}\,\circ\,\widetilde{\alpha}_{1})^{\ast} p^{\ast} \mathcal{O}_{X}(x^{\prime}) \,=\, \bigotimes_{j=1}^{m} (\widetilde{\pi}_{1}\,\circ\,\widetilde{\alpha}_{1})^{\ast} (\mathcal{O}_{Y}(y_{j}^{\prime})^{e_{j}})\,=\, (\widetilde{\pi}_{2}\circ\widetilde{\alpha}_{2})^{\ast}p^{\ast}\mathcal{O}_{X}(x^{\prime}).
\end{equation*}
As seen in Remark \ref{equal}, we have 
\begin{equation*} 
 c_{1}(\mathcal{L})\,=\, c_{1}(\mathcal{L}^{\prime}), \quad c_{1}(\widetilde{L}) \,=\, |\Gamma|c_{1}(\widetilde{\mathcal{L}})\quad \textrm{and} \quad c_{1}({\widetilde{L}^{\prime}})\,=\,|\Gamma| c_{1}(\widetilde{\mathcal{L}})
\end{equation*}
in $N^{1}(S)$ and $N^{1}(\widetilde{S})$, respectively. 

Therefore, in the study of the cones of divisors and cycles on $S$, it suffices to consider the numerical class associated with an unramified point $x\, \in\, X \setminus D$.

\subsection{Notation and conventions}

Following Convention~\ref{con1}, we shall not distinguish between an element $P\,\in\, N^{k}(S)$ and the corresponding cycle class $P\,\cap[S]\,\in\, N_{r_{1}+r_{2}-1-k}(S)$ (note that $\textrm{dim}(S)\,=\,r_{1}+r_{2}-1$), and we will use the same notation for both whenever the context is clear.

We fix the following notation. In $N_{r_{1}+r
_{2}-2}(S)$, define
\[
\delta_1
\,:=\,
c_1\!\bigl(\alpha_1^*\mathcal O_{\pe}(1)\bigr),
\qquad
\delta_2
\,:=\,
c_1\!\bigl(\alpha_2^*\mathcal O_{\pf}(1)\bigr),
\]
and
\[
\lambda
\,:=\,
c_1(\mathcal L),
\]
where \(\mathcal L\) is the line bundle defined in \eqref{l}.

Similarly, in $N^{1}(\widetilde S)\,=\, N_{r_{1}+r_{2}-2}(\widetilde S)$, define
\[
\widetilde\delta_1
\,:=\,
c_1\!\bigl(
\widetilde\alpha_1^*
\mathcal O_{\mathbb P(\widetilde E_1)}(1)
\bigr),
\qquad
\widetilde\delta_2
\,:=\,
c_1\!\bigl(
\widetilde\alpha_2^*
\mathcal O_{\mathbb P(\widetilde E_2)}(1)
\bigr),
\]
and
\[
\widetilde\lambda
\,:=\,
c_1(\widetilde{ L}).
\]

By \cite[Proposition~3.5]{BBM}, the special pullback satisfies
\begin{equation}\label{pp}
\widetilde p_{\mathrm{sp}}^{\,*}(\delta_1)
\,=\,
t_1\widetilde\delta_1,
\qquad
\widetilde p_{\mathrm{sp}}^{\,*}(\delta_2)
\,=\,
t_2\widetilde\delta_2,
\qquad
\widetilde p_{\mathrm{sp}}^{\,*}(\lambda)
\,=\,
\widetilde\lambda.
\end{equation}
Next, let
\[
\widetilde G
\,:=\,
c_{1}(\mathcal{\widetilde{L}})
\,\in\, N^{1}(\widetilde S).
\]
By construction,
\[
\widetilde\lambda
\,=\,
|\Gamma|\,\widetilde G
\]
in $N^{1}(\widetilde S)$.
\subsection{Pseudoeffective cone of fiber product}
Consider the pullback map $\textrm{Pic}(S)\ \xrightarrow{\,\,\,\widetilde{p}^{\,\ast}\,\,\,}\
\textrm{Pic}(\widetilde{S})$.
We have the following commutative diagram
\begin{equation}\label{dia} 
\begin{tikzcd} 
\textrm{CH}_{k}(\widetilde{S}) \arrow[r, "\widetilde{p}_{\ast}"] \arrow[d, ""] 
& \textrm{CH}_{k}(S) \arrow[d, ""] \\
{N}_{k}(\widetilde{S}) \arrow[r, "\widetilde{p}_{\ast}"] 
& {N}_{k}(S).
\end{tikzcd}
\end{equation}
Here, by abuse of notation, the induced map $${N}_{k}(\widetilde{S}) \, \longrightarrow \,
{N}_{k}(S)$$ is also denoted by $\widetilde{p}_{\ast}$.

\begin{lemma}\label{pushproduct}
Let $\widetilde p_* \,\colon\, N_k(\widetilde S)\,\longrightarrow\, N_k(S)$ be the pushforward map defined above. Then the following two statements hold:
\begin{enumerate}
\item For every \(k\,\in\,\{1,\,\cdots,\,r_1\,+\,r_2\,-\,2\}\), 
\[
\widetilde p_*(\widetilde\delta_1^k)
\,=\,
\frac{|\Gamma|}{t_1^k}\,\delta_1^k,
\qquad
\widetilde p_*(\widetilde\delta_2^k)
\,=\,
\frac{|\Gamma|}{t_2^k}\,\delta_2^k,
\]
and
\[
\widetilde p_*(\widetilde G)
\,=\,
\lambda.
\]

\item For every \(k_1,\,k_2\,\in\,\{1,\,\cdots,\,r_1\,+\,r_2\,-\,2\}\),
\[
\widetilde p_*
\bigl(
\widetilde\delta_1^{k_1}
\cdot\widetilde\delta_2^{k_2}
\bigr)
\,=\,
\frac{|\Gamma|}
{t_1^{k_1}t_2^{k_2}}
\,
\delta_1^{k_1}\cdot\delta_2^{k_2},
\]
and
\[
\widetilde p_*
\bigl(
\widetilde\delta_1^{k_1}\cdot
\widetilde\delta_2^{k_2}
\cdot\widetilde G
\bigr)
\,=\,
\frac{1}
{t_1^{k_1}t_2^{k_2}}
\,
\delta_1^{k_1}\cdot\delta_2^{k_2}\cdot\lambda.
\]
\end{enumerate}
\end{lemma}

\begin{proof}
By \eqref{pp}, we have
\[
\widetilde p_{\mathrm{sp}}^{\,*}(\delta_1)
\,=\,
t_1\widetilde\delta_1,
\qquad
\widetilde p_{\mathrm{sp}}^{\,*}(\delta_2)
\,=\,
t_2\widetilde\delta_2, \qquad \widetilde{p}^{\,\ast}_{\textrm{sp}}(\lambda)\,=\, \widetilde{\lambda}.
\]
Since \(\widetilde p_{\mathrm{sp}}^{\,*}\) is a ring homomorphism, it follows that
\[
\widetilde p_{\mathrm{sp}}^{\,*}(\delta_1^k)
\,=\,
t_1^k\widetilde\delta_1^k,
\qquad
\widetilde p_{\mathrm{sp}}^{\,*}(\delta_2^k)
\,=\,
t_2^k\widetilde\delta_2^k.
\]
Now using \eqref{comp} we obtain the following:
\[
|\Gamma|\,\delta_1^k
\,=\,
\widetilde p_*\!\bigl(
\widetilde p_{\mathrm{sp}}^{\,*}(\delta_1^k)
\bigr)
\,=\,
t_1^k\,\widetilde p_*(\widetilde\delta_1^k).
\]
Therefore,
\[
\widetilde p_*(\widetilde\delta_1^k)
\,=\,
\frac{|\Gamma|}{t_1^k}\,\delta_1^k.
\]
The same argument gives the following:
\[
\widetilde p_*(\widetilde\delta_2^k)
\,=\,
\frac{|\Gamma|}{t_2^k}\,\delta_2^k.
\]
Finally, since
\[
\widetilde p_{\mathrm{sp}}^{\,*}(\lambda)
\,=\,
\widetilde\lambda
\,=\,
|\Gamma|\,\widetilde G,
\]
by applying $\widetilde{p}_*$ to it we have
$$
|\Gamma|\lambda \ =\ \widetilde p_* \widetilde\lambda\ =\ |\Gamma|\widetilde p_*(\widetilde G).
$$
This implies that
\[
\widetilde p_*(\widetilde G)
\,=\,
\lambda.
\]
The identities in part $(2)$ are proved in exactly the same way as those in part $(1)$. We therefore omit the proof.
\end{proof}
The computation of the pseudoeffective cone of $S$ requires certain intersection-theoretic calculations. We begin by recalling the following relations from \cite{KMR}, which will be used in the proof of the next result:
\begin{equation}\label{re1}
\widetilde{\delta}_{1}^{r_{1}}\cdot \widetilde{G}\,=\, 0,\quad \widetilde{\delta}_{1}^{r_{1}+1}\,=\,0,  \quad \widetilde{\delta}_{2}^{r_{2}}\cdot \widetilde{G}\,=\, 0,\quad \widetilde{\delta}_{2}^{r_{2}+1}\,=\,0, 
\end{equation}
\begin{equation}\label{re2}
\widetilde{\delta}_{1}^{r_{1}}\,=\, (\textrm{deg}(\widetilde{E}_{1}))\, \widetilde{\delta}_{1}^{r_{1}-1}\cdot \widetilde{G},\quad \widetilde{\delta}_{2}^{r_{2}}\,=\, (\textrm{deg}(\widetilde{E}_{2}))\,  \widetilde{\delta}_{2}^{r_{2}-1}\cdot \widetilde{G}, 
\end{equation}
\begin{equation}\label{re3}
\widetilde{\delta}_{1}^{r_{1}}\cdot \widetilde{\delta}_{2}^{r_{2}-1}\,=\, \textrm{deg}(\widetilde{E}_{1}),\quad \widetilde{\delta}_{1}^{r_{1}-1}\cdot \widetilde{\delta}_{2}^{r_{2}}\,=\, \textrm{deg}(\widetilde{E}_{2}).
\end{equation}
\begin{corollary}\label{int}
The following relations hold in the numerical group \(N_{r_{1}+r_{2}-2}(S)\):
\[
\delta_1^{r_1+1}\,=\,0,\qquad
\delta_2^{r_2+1}\,=\,0,\qquad
\lambda^2\,=\,0,
\]
\[
\delta_1^{r_1}\cdot\lambda\,=\,0,
\qquad
\delta_2^{r_2}\cdot\lambda\,=\,0,
\]
and
\[
\delta_1^{r_1}
\,=\,
t_{1}\operatorname{par-deg}({E}_{1\ast})\,
\delta_1^{r_1-1}\cdot\lambda,
\qquad
\delta_2^{r_2}
\,=\,
t_{2}\operatorname{par-deg}({E}_{2\ast})\,
\delta_2^{r_2-1}\cdot\lambda.
\]
Consequently,
\[
\delta_1^{r_1}\cdot\delta_2^{r_2-1}
\,=\,
{\operatorname{par\text{-}deg}(E_{1*})} t_{1}^{r_{1}}t_{2}^{r_{2}-1},\qquad\delta_2^{r_2}\cdot\delta_1^{r_1-1}
\,=\,
{\operatorname{par\text{-}deg}(E_{2*})}t_{2}^{r_{2}}t_{1}^{r_{1}-1}.
\]
\end{corollary}

\begin{proof}
Applying Lemma \ref{pushproduct} to the relation $\widetilde{\delta}_{1}^{r_{1}+1}\,=\,0$ in $\eqref{re1}$, we obtain
\[
\widetilde{p}_{\ast}(\widetilde{\delta}_1^{\,r_1+1})
\,=\,
\frac{|\Gamma|}{t_1^{\,r_1+1}}\,
\delta_1^{\,{r_1}+1}.
\]
Hence $\delta_{1}^{r_{1}+1}\,=\,0$. By the same argument, $\delta_{2}^{r_{2}+{1}}\,=\,0$. Next, applying Lemma~\ref{pushproduct} to the relation $\widetilde{\delta}^{r_{1}}_{1}\cdot \widetilde{G}\,=\, 0$ in $\eqref{re1}$, we get 
\begin{equation*}
\widetilde{p}_{\ast}(\widetilde{\delta}_{1}^{r_{1}}\cdot \widetilde{G})\,=\, \frac{|\Gamma|}{t_{1}^{r_{1}}}\delta_{1}^{r_{1}}\cdot \lambda. \end{equation*}
Hence $\delta_{1}^{r_{1}}\cdot \lambda\,=\,0$. Similarly, $\delta_{2}^{r_{2}}\cdot \lambda\,=\,0$. To prove the next identity, apply $\widetilde{p}_{\ast}$ to the relation $\widetilde{\delta}_{1}^{r_{1}}\,=\, \textrm{deg}(\widetilde{E}_{1})\,\widetilde{\delta}_{1}^{r_{1}-1}\cdot \widetilde{G}$, which follows from \eqref{re2}. Using Lemma \ref{pushproduct}, we obtain
\begin{equation*}
\frac{|\Gamma|}{t_{1}^{r_{1}}}\delta_{1}^{r_{1}}\,=\,\frac{\textrm{deg}(\widetilde{E}_{1})}{t_{1}^{r_{1}-1}}\delta_{1}^{r_{1}-1}\cdot\lambda.    
\end{equation*}
Since $\textrm{deg}(\widetilde{E}_{1})\,=\,|\Gamma|{\operatorname{par\text{-}deg}(E_{1*})}$, cancelling $|\Gamma|$ gives 
\begin{equation*}
\delta_{1}^{r_{1}}\,=\, t_{1}\, \textrm{par-deg}({E}_{1\ast})\, \delta_{1}^{r_{1}-1}\cdot \lambda.    
\end{equation*}
Similarly, $\delta_{2}^{r_{2}}\,=\, t_{2}\, \textrm{par-deg}({E}_{2\ast})\, \delta_{1}^{r_{2}-1}\cdot \lambda$.  

The remaining two identities are obtained by applying $\widetilde{p}_{\ast}$ to the relations in \eqref{re3} and proceeding exactly as above.
\end{proof}

We now turn to the case where $E_{1*}$ and $E_{2*}$ are parabolic semistable vector bundles on $X$. Without loss of generality, we may assume that $r_{1}\,\leq\,r_{2}$. 

\begin{theorem}\label{stable}
Let \(E_{1*}\) and \(E_{2*}\) be parabolic semistable vector bundles on \(X\) of ranks \(r_1\) and \(r_2\) respectively, with parabolic slopes $\mu_{1}$ and $\mu_{2}$. Assume that $r_{1}\,\leq\,r_{2}$.  Then, for every $k\,\in\,\{1,\,\ldots\,,r_{1}\,+\,r_{2}\,-\,2\}$, 
the pseudoeffective cone $\overline{\mathrm{Eff}}_{k}(S)$ has the following description:
\[
\hspace{-2.7cm}\overline{\mathrm{Eff}}_{k}(X)
\,=\,
\begin{cases}
\left\langle
\Bigl\{
(\delta_{1}\,-\,\mu_{1}\lambda)^{i}
\cdot(\delta_{2}\,-\,\mu_{2}\lambda)^{r_{1}+r_{2}-1-k-i}
\Bigr\}_{i=0}^{\,r_{1}+r_{2}-1-k},
\;
\Bigl\{
\lambda\cdot\delta_{1}^{\,j}\cdot\delta_{2}^{\,r_{1}+r_{2}-k-j-2}
\Bigr\}_{j=0}^{\,r_{1}+r_{2}-k-2}
\right\rangle,
\qquad k\,\ge\, r_{2},
\\[4mm]

\left\langle
\Bigl\{
(\delta_{1}\,-\,\mu_{1}\lambda)^{i}
\cdot(\delta_{2}\,-\,\mu_{2}\lambda)^{r_{1}+r_{2}-1-k-i}
\Bigr\}_{i=0}^{\,r_{1}-1},
\;
\Bigl\{
\lambda\cdot\delta_{1}^{\,j}\cdot\delta_{2}^{\,r_{1}+r_{2}-k-j-2}
\Bigr\}_{j=0}^{\,r_{1}-1}
\right\rangle, \qquad r_{1}\,\le\, k\,\le\, r_{2}-1,
\\[4mm]

\left\langle
\Bigl\{
(\delta_{1}\,-\,\mu_{1}\lambda)^{i}
\cdot(\delta_{2}\,-\,\mu_{2}\lambda)^{r_{1}+r_{2}-1-k-i}
\Bigr\}_{i=k^{\prime}+1}^{\,r_{1}-1},
\;
\Bigl\{
\lambda\cdot\delta_{1}^{\,j}\cdot\delta_{2}^{\,r_{1}+r_{2}-k-j-2}
\Bigr\}_{j=k^{\prime}}^{\,r_{1}-1}
\right\rangle, \,\,\,  k\, <\, r_1,\,\, \, k^{\prime}\,=\,r_{1}-1-k.
\end{cases}
\]
\end{theorem}
\begin{proof}

Since $E_{1\ast}$ and $E_{2\ast}$ are parabolic semistable, Proposition \ref{semistable} implies that the corresponding orbifold bundles $\widetilde{E}_{1}$ and $\widetilde{E}_{2}$ are semistable. Let $\widetilde{\mu}_{1}$ and $\widetilde{\mu}_{2}$ denote the slopes of $\widetilde{E}_{1}$ and $\widetilde{E}_{2}$, respectively.

We first prove the theorem in the range $r_{1}\leq\,k\,\leq\,r_{2}-1$. By \cite[Theorem 3.2]{Kar}, the cone $\overline{\mathrm{Eff}}_k(\widetilde S)$ is generated by the classes 
\[
\overline{\textrm{Eff}}_k(\widetilde{S})
\,=\,
\Bigg\langle
\Big\{
(\widetilde{\delta}_1\,-\,\widetilde{\mu}_1\widetilde{G})^i
\cdot(\widetilde{\delta}_2\,-\,\widetilde{\mu}_2\widetilde{G})^{r_1+r_{2}-1-k-i}
\Big\}_{i=0}^{\,r_1-1}, \Big\{
\widetilde{G}\cdot
\widetilde{\delta}_1^{\,j}\cdot
\widetilde{\delta}_2^{\,r_1+r_2-k-j-2}
\Big\}_{j=0}^{\,r_1-1}
\Bigg\rangle.
\]
By Lemma \ref{pushproduct}, under the pushforward map 
\begin{equation*}
\widetilde{p}_{\ast} \,:\, \overline{\textrm{Eff}}_k(\widetilde{S})\,\longrightarrow\,\overline{\textrm{Eff}}_k(S),
\end{equation*}
we have
\begin{equation*}
\widetilde p_*
\left((\widetilde{\delta}_1\,-\,\widetilde{\mu}_1\widetilde G)^i\right)
\,=\,
\frac{|\Gamma|}{t_1^i}
(\delta_1\,-\,t_1\mu_1\lambda)^i,
\end{equation*}
\begin{equation*}
\widetilde p_*
\left((\widetilde{\delta}_2\,-\,\widetilde{\mu}_2\widetilde G)^{r_1+r_2-1-k-i}\right)
\,=\,
\frac{|\Gamma|}{t_2^{\,r_1+r_2-1-k-i}}
(\delta_2\,-\,t_2\mu_2\lambda)^{r_1+r_2-1-k-i},
\end{equation*}
and 
\begin{equation*}
\widetilde{p}_{\ast}(\widetilde{G}\cdot\widetilde{\delta}_{1}^{i}\cdot\widetilde{\delta}_{2}^{r_{1}+r_{2}-1-k-j-2})\,=\,\frac{1}{t_{1}^{i}t_{2}^{r_{1}+r_{2}-1-k-j-2}} (\lambda\cdot\delta_{1}^{i}\cdot \delta_{2}^{r_{1}+r_{2}-1-k-j-2}),  
\end{equation*}
for all $i,j\,\in\, \{0,1,\,\cdots,\,r_{1}-1\}$. It follows that
\begin{equation*}
\Bigg\langle
\Big\{
(\delta_1\,-\,t_1\mu_1\lambda)^i
\cdot(\delta_2\,-\,t_2\mu_2\lambda)^{r_1+r_2-1-k-i}
\Big\}_{i=0}^{\,r_1-1}, \Big\{
\lambda\cdot
\delta_1^{\,j}\cdot
\delta_2^{\,r_1+r_2-k-j-2}
\Big\}_{j=0}^{\,r_1-1}
\Bigg\rangle\,\subseteq\, \overline{\textrm{Eff}}_k(S).
\end{equation*}
By \cite[Corollary~3.8]{FL1}, none of the above classes is numerically trivial. We now show that these classes are numerically independent. For each $i\,\in\, \{0,\,\cdots,\, r_{1}\,-\,1\}$, define
\[
M_i
\,:=\,
\lambda\cdot
(\delta_1\,-\,t_1\mu_1\lambda)^{r_1-1-i}\cdot
(\delta_2\,-\,t_2\mu_2\lambda)^{k+i-r_1}
\,\in\, N^k(S).
\]
Let 
$$\alpha_{i}\,=\,(\delta_1\,-\,t_1\mu_1\lambda)^i
\cdot(\delta_2\,-\,t_2\mu_2\lambda)^{r_1+r_2-1-k-i},$$
and 
\begin{equation*}
\beta_{j}\,=\,\widetilde{G}\cdot
\widetilde{\delta}_1^{\,j}\cdot
\widetilde{\delta}_2^{\,r_1+r_2-k-j-2} .   
\end{equation*}

Clearly,
\[
\alpha_m\cdot M_i
\,=\,
\begin{cases}
0,
&
m\,\neq\, i,
\\[2mm]
\lambda\cdot
(\delta_1\,-\,t_1\mu_1\lambda)^{r_1-1}\cdot
(\delta_2\,-\,t_2\mu_2\lambda)^{r_2-1}
\,\neq\, 0,
&
m\,=\,i,
\end{cases}
\]
and
\[
\beta_j\cdot M_i\,=\,0
\]
for all \(i\) and \(j\). Both assertions follow from Corollary~\ref{int}. Consequently, the classes $\alpha_{0},\,\cdots,\,\alpha_{r_{1}-1},\, \beta_{0},\,\cdots,\, \beta_{r_{1}-1}$ are numerically independent. Therefore, by \cite[Corollary 3.22]{FL2}, these classes generate \(\overline{\mathrm{Eff}}_k(S)\) for $r_{1}\,\leq\,k\,\leq\,r_{2}-1$. The remaining cases can be proved analogously.
\end{proof}
Let $E_{1*}$ and $E_{2*}$ be parabolic unstable vector bundles on $X$,
with underlying vector bundles $E_1$ and $E_2$ respectively, and with parabolic divisor $D$. For
$j\,\in\,\{1,2\}$, let
\[
E_j
\,=\,
E_j^0
\,\supsetneq\,
E_j^1
\,\supsetneq\,
\cdots
\,\supsetneq\,
E_j^{l_j}
\,=\,
0
\]
be the Harder--Narasimhan filtration of $E_{j*}$. Define
\[
Q_{j*}^i
\,:=\,
\bigl(E_j^{i-1}/E_j^i\bigr)_*,
\qquad
i\,=\,1,\,\cdots,\,l_j.
\]
For each $i\,\in\,\{1,\,\cdots,\,l_j\}$, set
\begin{equation}\label{n}
n_{ji}
\,:=\,
\operatorname{rk}(Q_{j*}^i),
\qquad
d_{ji}
\,:=\,
\operatorname{par\text{-}deg}(Q_{j*}^i),
\qquad
\mu_{ji}
\,:=\,
\frac{d_{ji}}{n_{ji}}.
\end{equation}
Then each $Q_{j*}^i$ is parabolic  semistable, and
\[
\mu_{j1}
\,>\,
\mu_{j2}
\,>\,
\cdots
\,>\,
\mu_{jl_j}.
\]
Let $\widetilde E_1$ and $\widetilde E_2$ be the orbifold bundles corresponding to
$E_{1*}$ and $E_{2*}$, respectively. By the correspondence between parabolic
bundles and orbifold bundles established in \cite{Bi1}, the Harder--Narasimhan
filtrations correspond. Consequently, for each $j\,\in\,\{1,2\}$, we have
\[
\widetilde E_j
\,=\,
\widetilde E_j^0
\,\supsetneq\,
\widetilde E_j^1
\,\supsetneq\,
\cdots
\,\supsetneq\,
\widetilde E_j^{l_j}
\,=\,
0,
\]
whose successive quotients
\[
\widetilde Q_j^i
\,:=\,
\widetilde E_j^{i-1}/\widetilde E_j^i,
\qquad
i\,=\,1,\,\cdots,\,l_j,
\]
are semistable.
For each $i\,\in\,\{1,\,\cdots,\,l_j\}$, define
\begin{equation*}
\widetilde n_{ji}
\,:=\,
\operatorname{rk}(\widetilde Q_j^i),
\qquad
\widetilde d_{ji}
\,:=\,
\deg(\widetilde Q_j^i),
\qquad
\widetilde\mu_{ji}
\,:=\,
\mu(\widetilde Q_j^i)
\,=\,
\frac{\widetilde d_{ji}}{\widetilde n_{ji}},
\end{equation*}
such that
\[
\widetilde{\mu}_{j1}
\,>\,
\widetilde{\mu}_{j2}
\,>\,
\cdots
\,>\,
\widetilde{\mu}_{jl_j}.
\]
Consider the natural inclusion
\begin{equation}\label{embo}
\widetilde{i}
\,=\,
\widetilde{i}_{1}\,\times\,\widetilde{i}_{2}
\,:\,
\mathbb P(\widetilde Q_{1}^{1})
\,\times_Y\,
\mathbb P(\widetilde Q_{2}^{1})
\,\hookrightarrow\,
\mathbb P(\widetilde E_{1})
\,\times_Y\,
\mathbb P(\widetilde E_{2}),
\end{equation}
induced by the natural inclusions
\[
\widetilde{i}_{1}
\,:\,
\mathbb P(\widetilde Q_{1}^{1})
\,\hookrightarrow\,
\mathbb P(\widetilde E_{1})
\qquad\text{and}\qquad
\widetilde{i}_{2}
\,:\,
\mathbb P(\widetilde Q_{2}^{1})
\,\hookrightarrow\,
\mathbb P(\widetilde E_{2}).
\]
The restriction of the quotient map \(\widetilde p\) in \eqref{prod} to $\mathbb P(\widetilde Q_{1}^{1})
\,\times_Y\,
\mathbb P(\widetilde Q_{2}^{1})$ induces an isomorphism
\[
\widetilde p\Big(
\mathbb P(\widetilde Q_{1}^{1})
\,\times_Y\,
\mathbb P(\widetilde Q_{2}^{1})
\Big)
\,\cong\,
\Big(
\mathbb P(\widetilde Q_{1}^{1})
\,\times_Y\,
\mathbb P(\widetilde Q_{2}^{1})
\Big)/\Gamma \,=\, \mathbb P(Q_{1*}^{1})
\,\times_X\,
\mathbb P(Q_{2*}^{1}).
\]
Hence
\begin{equation}\label{im}
\widetilde p\Big(
\mathbb P(\widetilde Q_{1}^{1})
\,\times_Y\,
\mathbb P(\widetilde Q_{2}^{1})
\Big)
\,=\,
\mathbb P(Q_{1*}^{1})
\,\times_X\,
\mathbb P(Q_{2*}^{1}).
\end{equation}
Consequently, there is a natural closed embedding
\begin{equation}\label{emb}
i\,:\,
\mathbb P(Q_{1*}^{1})
\,\times_X\,
\mathbb P(Q_{2*}^{1})
\,\hookrightarrow\,
\mathbb P(E_{1*})
\,\times_X\,
\mathbb P(E_{2*}).
\end{equation}
By \cite[Example~3.2.17]{Fu}, adjusted to bundles of quotients over curves, we have
\[
[\mathbb P(\widetilde Q_{1}^{1})]
\,=\,
\widetilde\delta_1^{\,r_1-n_{11}}
\,+\,
(\widetilde d_{11}\,-\,\widetilde d_1)
\widetilde\delta_1^{\,r_1-n_{11}-1}\cdot\widetilde G,
\]
and
\[
[\mathbb P(\widetilde Q_{2}^{1})]
\,=\,
\widetilde\delta_2^{\,r_2-n_{21}}
\,+\,
(\widetilde d_{21}\,-\,\widetilde d_2)
\widetilde\delta_2^{\,r_2-n_{21}-1}\cdot\widetilde G.
\]
Since
\[
\Big[
\mathbb P(\widetilde Q_{1}^{1})
\,\times_Y\,
\mathbb P(\widetilde Q_{2}^{1})
\Big]
\,=\,
\widetilde\alpha_1^{*}
\bigl[
\mathbb P(\widetilde Q_{1}^{1})
\bigr]
\cdot
\widetilde\alpha_2^{*}
\bigl[
\mathbb P(\widetilde Q_{2}^{1})
\bigr],
\]
it follows that
\begin{equation}\label{class}
\Big[
\mathbb P(\widetilde Q_{1}^1)
\,\times_Y\,
\mathbb P(\widetilde Q_{2}^1)
\Big]
\ =
\end{equation}
$$
\left(
\widetilde\delta_1^{\,r_1-n_{11}}
\,+\,
(\widetilde d_{11}\,-\,\widetilde d_1)
\widetilde\delta_1^{\,r_1-n_{11}-1}\cdot\widetilde G
\right)\cdot
\left(
\widetilde\delta_2^{\,r_2-n_{21}}
\,+\,
(\widetilde d_{21}\,-\,\widetilde d_2)
\widetilde\delta_2^{\,r_2-n_{21}-1}\cdot\widetilde G
\right).
$$

We note that the statement of \cite[Theorem~3.3]{Kar} contains a mistake in the first case. The relation between the parameters $k$ and $k'$ (denoted by $t$ in \cite{Kar}) is missing from the statement. More precisely, one should assume the following:
\begin{equation*}
k\,=\,n_{11}\,-\,1\,-\,k',
\qquad
k'\,=\,0,\,\cdots,\,n_{11}\,-\,2.
\end{equation*}
Without this condition, the indexing of the generators is incomplete. For the sake of completeness, we state the corrected version of the theorem and provide a proof.
\begin{theorem}\label{ogen}
Let $\widetilde{E}_{1}$ and $\widetilde{E}_{2}$ be unstable vector bundles of ranks
$r_1$ and $r_2$, and degrees $\widetilde d_1$ and $\widetilde d_2$,
respectively, over a smooth projective curve $Y$. Set $\widetilde S\,:=\,\mathbb P(\widetilde E_{1})\,\times_Y\,\mathbb P(\widetilde{E}_{2}).$ Without loss of generality, assume that $n_{11}\,\leq\,n_{21}$, and define $\textbf{n}\,:=\, \emph{dim}[\mathbb{P}(\widetilde{Q}_{1}^{1})\,\times\,\mathbb{P}(\widetilde{Q}_{2}^{1})]\,=\, n_{11}\,+\,n_{21}-1$. Then, for every $k\,\in\,\{1,\,\cdots,\, \mathbf{n}\}, $ pseudoeffective cone $\overline{\operatorname{Eff}}_k(\widetilde S)$ is described as follows:

\medskip
\[
\hspace{-2.5cm}\overline{\operatorname{Eff}}_k(\widetilde S)\,=\,
\begin{cases}

\Bigg\langle
\Big\{
[\mathbb P(\widetilde Q_1^1)\,\times_Y\, \mathbb P(\widetilde Q_2^1)]\cdot
(\widetilde\delta_1\,-\,\widetilde\mu_{11}\widetilde G)^i
\cdot(\widetilde\delta_2\,-\,\widetilde\mu_{21}\widetilde G)^{\mathbf n-k-i}
\Big\}_{i=k^{\prime}+1}^{n_{11}-1},
\\[0.5em]
\hspace{2em}
\Big\{
\widetilde{G} \cdot
\widetilde\delta_1^{\,r_1-n_{11}+j}\cdot
\widetilde\delta_2^{\,r_2+n_{11}-k-j-2}
\Big\}_{j=k^{\prime}}^{n_{11}-1}
\Bigg\rangle,
&
\begin{aligned}
k&\,=\,n_{11}\,-\,1\,-\,k^{\prime},\\
k^{\prime}&\,=\,0,\,\cdots,\,n_{11}\,-\,2;
\end{aligned}
\\[2em]

\Bigg\langle
\Big\{
[\mathbb P(\widetilde Q_1^1)\,\times_Y\, \mathbb P(\widetilde Q_2^1)]\cdot
(\widetilde\delta_1\,-\,\widetilde\mu_{11}\widetilde G)^i
\cdot(\widetilde\delta_2\,-\,\widetilde\mu_{21}\widetilde G)^{\mathbf n-k-i}
\Big\}_{i=0}^{n_{11}-1},
\\[0.5em]
\hspace{2em}
\Big\{
\widetilde G\cdot
\widetilde\delta_1^{\,r_1-n_{11}+j}\cdot
\widetilde\delta_2^{\,r_2+n_{11}-k-j-2}
\Big\}_{j=0}^{n_{11}-1}
\Bigg\rangle,
&
n_{11}\,\le\, k\,<\,n_{21};
\\[2em]

\Bigg\langle
\Big\{
[\mathbb P(\widetilde Q_1^1)\,\times_Y\, \mathbb P(\widetilde Q_2^1)]\cdot
(\widetilde\delta_1\,-\,\widetilde\mu_{11}\widetilde G)^i
\cdot(\widetilde\delta_2\,-\,\widetilde\mu_{21}\widetilde G)^{\mathbf n-k-i}
\Big\}_{i=0}^{\mathbf n-k},
\\[0.5em]
\hspace{2em}
\Big\{
\widetilde G \cdot
\widetilde\delta_1^{\,r_1-n_{11}+j}\cdot
\widetilde\delta_2^{\,r_2+n_{11}-k-j-2}
\Big\}_{j=0}^{\mathbf n-k}
\Bigg\rangle,
&
k\,\ge\, n_{21}.
\end{cases}
\]
\end{theorem}
\begin{proof}
We prove the theorem in the range
\[
k\,=\,n_{11}\,-\,1\,-\,k^{\prime},
\qquad
k^{\prime}\,=\,0,1,\,\cdots,\,n_{11}-2.
\]
For 
$i\,=\,k^{\prime}\,+\,1,\,\cdots,\,n_{11}\,-\,1$ and $j\,=\, k^{\prime},\,\cdots,\, n_{11}\,-\,1$, consider the classes
\[
\phi_i
\,:=\,
[\mathbb{P}(\widetilde{Q}_{1}^{1})\,\times_Y\,
\mathbb{P}(\widetilde{Q}_{2}^{1})]\cdot
(\widetilde{\delta}_1\,-\,\widetilde{\mu}_{11}\widetilde G)^i\cdot
(\widetilde{\delta}_2\,-\,\widetilde{\mu}_{21}\widetilde G)^{\mathbf n-k-i},
\]
\[
\psi_{j}\ =\ \widetilde{G}\cdot
\widetilde\delta_1^{\,r_1-n_{11}+j}\cdot
\widetilde\delta_2^{\,r_2+n_{11}-k-j-2}.
\]
Clearly, $\widetilde{\delta}_{1}\,-\,\widetilde{\mu}_{11}\widetilde{G}$ and $\widetilde{\delta}_{2}\,-\,\widetilde{\mu}_{21}\widetilde{G}$ are nef divisors on $\widetilde{S}$. Therefore, $\phi_{i}\,\in\, \overline{\operatorname{Eff}}_k(\widetilde S)$. Since $\widetilde G$ is the class of a fiber of the projection
$\widetilde S\,\longrightarrow\, Y$, it is effective. Moreover, the restrictions of
$\widetilde\delta_1$ and $\widetilde\delta_2$ to a fiber
$\mathbb P^{r_1-1}\,\times\,\mathbb P^{r_2-1}$ are the hyperplane classes on the two factors. Hence $\psi_{j}$
is represented by an effective cycle on a fiber and therefore pseudoeffective. We claim that $\phi_i$ and $\psi_{j}$ lies on the boundary of
$\overline{\operatorname{Eff}}_k(\widetilde S)$. 
Define
\[
D_i
\,:=\,
(\widetilde{\delta}_1\,-\,\widetilde{\mu}_{11}\widetilde G)^{n_{11}-i}\cdot
(\widetilde{\delta}_2\,-\,\widetilde{\mu}_{21}\widetilde G)^{k+i-n_{11}},
\]
and
\[
D'_j
\,:=\,
\widetilde G\cdot
(\widetilde{\delta}_{1}\,-\,\widetilde{\mu}_{11}\widetilde G)^{j-k'}\cdot
(\widetilde{\delta}_{2}\,-\,\widetilde{\mu}_{21}\widetilde G)^{n_{11}-2-j}.
\]
Since \(\widetilde G\),
\(\widetilde{\delta}_1-\widetilde{\mu}_{11}\widetilde G\), and
\(\widetilde{\delta}_2-\widetilde{\mu}_{21}\widetilde G\) are nef divisor classes, it follows that \(D_i\) and \(D'_j\) are nonzero nef classes.
We first compute
\[
\phi_i\cdot D_i
\,=\,
[\mathbb P(\widetilde Q_1^1)\,\times_Y\, \mathbb P(\widetilde Q_2^1)]\cdot
(\widetilde{\delta}_1\,-\,\widetilde{\mu}_{11}\widetilde G)^{n_{11}}\cdot
(\widetilde{\delta}_2\,-\,\widetilde{\mu}_{21}\widetilde G)^{n_{21}-1}.
\]
Substituting the expression obtained in \eqref{class}, we get
\begin{equation}\label{phi}
\phi_i\cdot D_i
\,=\,
\widetilde{\delta}_{1}^{r_{1}}
\cdot\widetilde{\delta}_{2}^{r_{2}-1}
\,-\,\widetilde d_{1}\,
\widetilde{\delta}_{1}^{r_{1}-1}\cdot
\widetilde{\delta}_{2}^{r_{2}-1}\cdot
\widetilde G
\,+\,
(\widetilde d_{21}\,-\,\widetilde d_{2})\,
\widetilde{\delta}_{1}^{r_{1}}
\cdot\widetilde{\delta}_{2}^{r_{2}-2}\cdot\widetilde G.
\end{equation}
Since \(\dim \mathbb P(\widetilde E_1)\,=\,r_1\), we have
\begin{equation}\label{van}
\widetilde{\delta}_1^{r_1}\cdot\widetilde G\,=\,0.
\end{equation}
Moreover, by \cite[\S~3]{KMR},
\begin{equation}\label{dot}
\widetilde{\delta}_1^{r_1}
\,=\,
\widetilde d_1\,
\widetilde{\delta}_1^{r_1-1}\cdot\widetilde G.
\end{equation}
Substituting the identities \eqref{dot} and \eqref{van} into \eqref{phi}, we obtain
\[
\phi_i\cdot D_i\,=\,0.
\]
Similarly,
\[
\begin{aligned}
\psi_j\cdot D'_j
&=
\widetilde G\cdot
\widetilde\delta_1^{\,r_1-n_{11}+j}\cdot
\widetilde\delta_2^{\,r_2+n_{11}-k-j-2}
\cdot
\widetilde G\cdot
(\widetilde{\delta}_{1}\,-\,\widetilde{\mu}_{11}\widetilde G)^{j-k'}\cdot
(\widetilde{\delta}_{2}\,-\,\widetilde{\mu}_{21}\widetilde G)^{n_{11}-2-j}
\\
&=
\widetilde G^2\cdot
\widetilde\delta_1^{\,r_1-n_{11}+j}\cdot
\widetilde\delta_2^{\,r_2+n_{11}-k-j-2}\cdot
(\widetilde{\delta}_{1}-\widetilde{\mu}_{11}\widetilde G)^{j-k'}\cdot
(\widetilde{\delta}_{2}-\widetilde{\mu}_{21}\widetilde G)^{n_{11}-2-j}.
\end{aligned}
\]
Since \(\widetilde G^2\,=\,0\), it follows that
\[
\psi_j\cdot D'_j\,=\,0.
\]
Hence both \(\phi_i\) and \(\psi_j\) lie on the boundary of
\(\overline{\operatorname{Eff}}_k(\widetilde S)\). The remainder of the argument is identical to that of \cite[Theorem~3.2]{Kar}.
\end{proof}

\begin{theorem}\label{usp}
Let $E_{1*}$ and $E_{2*}$ be unstable parabolic vector bundles of ranks $r_1$ and $r_2$, respectively, over a smooth projective curve $X$. Set $S\,:=\,\mathbb P(E_{1*})\,\times_X\,\mathbb P(E_{2*}).$ With notations as in \eqref{n}, assume
without loss of generality that $n_{11}\,\leq\, n_{21}$, and define $\mathbf n\,:=\,n_{11}\,+\,n_{21}\,-\,1.$ Then, for every $k\,\in\,\{1,\,\cdots,\,\mathbf n\}$, the pseudoeffective cone $\overline{\emph{Eff}}_{k}(S)$ is given by the following description:
\[
\hspace{-2.5cm}\overline{\emph{Eff}}_k(S)\,:=\,
\begin{cases}
\Bigg\langle
\Bigl\{
[\mathbb P(Q_{1*}^{1})\,\times_X\,\mathbb P(Q_{2*}^{1})]\cdot
(\delta_1\,-\,t_1\mu_{11}\lambda)^i\cdot
(\delta_2\,-\,t_2\mu_{21}\lambda)^{\mathbf n-k-i}
\Bigr\}_{i=k^{\prime}+1}^{n_{11}-1},
\\[2mm]
\hspace{1.5cm}
\Bigl\{
\lambda\cdot
\delta_1^{\,r_1-n_{11}+j}
\cdot
\delta_2^{\,r_2+n_{11}-k-j-2}
\Bigr\}_{j=k^{\prime}}^{n_{11}-1}
\Bigg\rangle,
&
\begin{array}{l}
k\,=\,n_{11}\,-\,1\,-\,k',\\
k'\,=\,0,\,\cdots,\,n_{11}\,-\,2;
\end{array}
\\[8mm]

\Bigg\langle
\Bigl\{
[\mathbb P(Q_{1*}^{1})\,\times_X\,\mathbb P(Q_{2*}^{1})]\cdot
(\delta_1\,-\,t_1\mu_{11}\lambda)^i
\cdot(\delta_2\,-\,t_2\mu_{21}\lambda)^{\mathbf n-k-i}
\Bigr\}_{i=0}^{n_{11}-1},
\\[2mm]
\hspace{1.5cm}
\Bigl\{
\lambda\cdot
\delta_1^{\,r_1-n_{11}+j}
\cdot
\delta_2^{\,r_2+n_{11}-k-j-2}
\Bigr\}_{j=0}^{n_{11}-1}
\Bigg\rangle,
&
n_{11}\,\le\, k\,<\,n_{21};
\\[8mm]

\Bigg\langle
\Bigl\{
[\mathbb P(Q_{1*}^{1})\,\times_X\,\mathbb P(Q_{2*}^{1})]\cdot
(\delta_1\,-\,t_1\mu_{11}\lambda)^i\cdot
(\delta_2\,-\,t_2\mu_{21}\lambda)^{\mathbf n-k-i}
\Bigr\}_{i=0}^{\mathbf n-k},
\\[2mm]
\hspace{1.5cm}
\Bigl\{
\lambda\cdot
\delta_1^{\,r_1-n_{11}+j}
\cdot
\delta_2^{\,r_2+n_{11}-k-j-2}
\Bigr\}_{j=0}^{\mathbf n-k}
\Bigg\rangle,
&
k\,\ge\, n_{21}.
\end{cases}
\]
Furthermore, the pushforward map $i_* \,:\,
\overline{\mathrm{Eff}}_{k}
\bigl(
\mathbb P(Q_{1*}^{1})
\,\times_X\,
\mathbb P(Q_{2*}^{1})
\bigr)
\,\longrightarrow\,
\overline{\mathrm{Eff}}_{k}(S)$, 
induced by the embedding $i$ in \eqref{emb}, is an isomorphism for all $k\,<\,\textbf{n}$.
\end{theorem}

\begin{proof}
We first compute the class $[
\mathbb P(Q_{1*}^{1})
\,\times_X\,
\mathbb P(Q_{2*}^{1})
]$ in terms of the classes \(\delta_1\), \(\delta_2\), and \(\lambda\). Define
\[
Z
\,:=\,
\left(
\delta_1^{\,r_1-n_{11}}
\,+\,
(d_{11}\,-\,d_1)t_1
\delta_1^{\,r_1-n_{11}-1}\cdot\lambda
\right)
\cdot
\left(
\delta_2^{\,r_2-n_{21}}
\,+\,
(d_{21}\,-\,d_2)t_2
\delta_2^{\,r_2-n_{21}-1}\cdot\lambda
\right).
\]
Clearly, by \eqref{pp}, we have
\[
\begin{aligned}
\widetilde p_{\mathrm{sp}}^{\,*}(Z)
&\,=\,
t_1^{\,r_1-n_{11}}
t_2^{\,r_2-n_{21}}
\left(
\widetilde\delta_1^{\,r_1-n_{11}}
\,+\,
(d_{11}\,-\,d_1)
\widetilde\delta_1^{\,r_1-n_{11}-1}\cdot
\widetilde\lambda
\right)\cdot
\left(
\widetilde\delta_2^{\,r_2-n_{21}}
\,+\,
(d_{21}\,-\,d_2)
\widetilde\delta_2^{\,r_2-n_{21}-1}\cdot
\widetilde\lambda
\right).
\end{aligned}
\]
Let
\[
\widetilde Z
\,=\,
\left(
\widetilde{\delta}_{1}^{r_{1}-n_{11}}
\,+\,
(d_{11}\,-\,d_{1})
\widetilde{\delta}_{1}^{r_{1}-n_{11}-1}\cdot\widetilde{\lambda}
\right)
\cdot
\left(
\widetilde{\delta}_{2}^{r_{2}-n_{21}}
\,+\,
(d_{21}\,-\,d_{2})
\widetilde{\delta}_{2}^{r_{2}-n_{21}-1}\cdot\widetilde{\lambda}
\right).
\]
Using the identities 
\[
\widetilde{\lambda}\,=\,|\Gamma|\widetilde G,
\qquad
\widetilde d_{ij}\,=\,|\Gamma|d_{ij},
\qquad
\widetilde d_i\,=\,|\Gamma|d_i,
\]
we obtain
\[
\widetilde Z
\,=\,
\left(
\widetilde{\delta}_{1}^{r_{1}-n_{11}}
+
(\widetilde d_{11}-\widetilde d_{1})
\widetilde{\delta}_{1}^{r_{1}-n_{11}-1}\cdot\widetilde G
\right)
\cdot
\left(
\widetilde{\delta}_{2}^{r_{2}-n_{21}}
+
(\widetilde d_{21}-\widetilde d_{2})
\widetilde{\delta}_{2}^{r_{2}-n_{21}-1}\cdot\widetilde G
\right).
\]
Now, Lemma~\ref{pushproduct} implies that
\begin{equation}\label{eq1}
\widetilde p_*(\widetilde Z)
\,=\,
\frac{|\Gamma|^{2}}
{t_1^{r_1-n_{11}}t_2^{r_2-n_{21}}}
\,Z.
\end{equation}
On the other hand,
\begin{equation}\label{eq2}
\widetilde p_*
\Big(
\big[
\mathbb P(\widetilde Q_{1}^{1})
\,\times_Y\,
\mathbb P(\widetilde Q_{2}^{1})
\big]
\Big)
\,=\,
|\Gamma|\,
\big[
\mathbb P(Q_{1*}^{1})
\,\times_X\,
\mathbb P(Q_{2*}^{1})
\big].
\end{equation}
Clearly, by \eqref{class}, we have $\widetilde Z
\,=\,
\big[
\mathbb P(\widetilde Q_{1}^{1})
\times_Y
\mathbb P(\widetilde Q_{2}^{1})
\big].$ Therefore, by \eqref{eq1} and \eqref{eq2}, we obtain
\[
|\Gamma|\,
\big[
\mathbb P(Q_{1*}^{1})
\,\times_X\,
\mathbb P(Q_{2*}^{1})
\big]
\,=\,
\frac{|\Gamma|^{2}}
{t_1^{r_1-n_{11}}t_2^{r_2-n_{21}}}
\,Z.
\]
Therefore,
\begin{equation}\label{pgen1}
\big[
\mathbb P(Q_{1*}^{1})
\times_X
\mathbb P(Q_{2*}^{1})
\big]
=
\frac{|\Gamma|\, Z}
{t_1^{r_1-n_{11}}t_2^{r_2-n_{21}}}.
\end{equation}
To compute the pseudoeffective cone of $S$, we first consider the case:
\begin{equation}\label{case}
n_{11}\,\le\, n_{21},
\qquad
k\,=\,n_{11}\,-\,1\,-\,k',
\qquad
k'\,=\,0,\,\cdots\,,n_{11}\,-\,2.
\end{equation}
The arguments in the remaining cases are analogous.

We claim that the classes
\[
\Big\{
\big[
\mathbb P(Q_{1*}^{1})
\,\times_X\,
\mathbb P(Q_{2*}^{1})
\big]\cdot
(\delta_1\,-\,t_1\mu_{11}\lambda)^i\cdot
(\delta_2\,-\,t_2\mu_{21}\lambda)^{\mathbf n-k-i}
\Big\}_{i=k'+1}^{n_{11}-1},
\]
and
\[
\Big\{
\lambda\cdot
\delta_1^{r_1-n_{11}+j}\cdot
\delta_2^{r_2+n_{11}-k-j-2}
\Big\}_{j=k'}^{n_{11}-1},
\]
generate \(\overline{\mathrm{Eff}}_k(S)\).  By Theorem~\ref{ogen}, and under the assumptions in \eqref{case}, the cone \(\overline{\mathrm{Eff}}_k(\widetilde S)\) is generated by the following classes:
\[
\Big\{
\big[
\mathbb P(\widetilde Q_{1}^{1})
\,\times_Y\,
\mathbb P(\widetilde Q_{2}^{1})
\big]\cdot
(\widetilde\delta_1\,-\,\widetilde\mu_{11}\widetilde G)^i
\cdot(\widetilde\delta_2\,-\,\widetilde\mu_{21}\widetilde G)^{\mathbf n-k-i}
\Big\}_{i=k'+1}^{n_{11}-1},
\]
and
\[
\Big\{
\widetilde G\cdot
\widetilde\delta_1^{r_1-n_{11}+j}\cdot
\widetilde\delta_2^{r_2+n_{11}-k-j-2}
\Big\}_{j=k'}^{n_{11}-1}.
\]
Since the quotient map $\widetilde p\,:\,\widetilde S\,\longrightarrow\, S$ induces a pushforward map

\[
\overline{\mathrm{Eff}}_k(\widetilde S)
\,\xlongrightarrow{\widetilde{p}_{\ast}}\,
\overline{\mathrm{Eff}}_k(S),
\]
it follows from Lemma \ref{pushproduct} that
\[
\Bigg\langle
\Big\{
\big[
\mathbb P(Q_{1*}^{1})
\,\times_X\,
\mathbb P(Q_{2*}^{1})
\big]\cdot
(\delta_1\,-\,t_1\mu_{11}\lambda)^i
\cdot(\delta_2\,-\,t_2\mu_{21}\lambda)^{\mathbf n-k-i}
\Big\}_{i=k'+1}^{n_{11}-1}\, ,
\]
\[
\Big\{
\lambda\cdot
\delta_1^{r_1-n_{11}+j}\cdot
\delta_2^{r_2+n_{11}-k-j-2}
\Big\}_{j=k'}^{n_{11}-1}
\Bigg\rangle
\ \subseteq\ \overline{\mathrm{Eff}}_k(S).
\]
By \cite[Corollary~3.8]{FL1}, none of the above classes is numerically trivial. We shall now prove that they are numerically independent. To this end, for each \(m\,\in\,\{k'\,+\,1,\,\cdots\,,n_{11}\,-\,1\}\), and $k'\,=\,0,\,\cdots\,,n_{11}\,-\,2$. We define the class
\[
M_m
\,:=\,
\lambda\cdot
(\delta_1\,-\,t_1\mu_{11}\lambda)^{n_{11}-1-m}
\cdot
(\delta_2\,-\,t_2\mu_{21}\lambda)^{k+m-n_{11}}
\,\in\, N^k(S).
\]
For \(i\,=\,k'\,+\,1,\,\cdots\,,n_{11}\,-\,1\), define
\[
\phi_i
\,:=\,
\bigl[
\mathbb P(Q_{1*}^{1})
\,\times_X\,
\mathbb P(Q_{2*}^{1})
\bigr]\cdot
(\delta_1\,-\,t_1\mu_{11}\lambda)^i\cdot
(\delta_2\,-\,t_2\mu_{21}\lambda)^{\mathbf n-k-i}.
\]
Similarly, for \(j\,=\,k',\,\cdots,\,n_{11}\,-\,1\), define
\[
\psi_j
\,:=\,
\lambda\cdot
\delta_1^{r_1-n_{11}+j}\cdot
\delta_2^{r_2+n_{11}-k-j-2}.
\]
A direct computation shows that
$$\phi_i\cdot M_m\,=\,\delta_1^{\,r_1+i-m-1}\cdot
\delta_2^{\,r_2+m-i-1}\cdot
\lambda.$$
If \(i\,=\,m\), then
\begin{equation}\label{i}
\phi_i\cdot M_i
\,=\,
\delta_1^{\,r_1-1}\cdot
\delta_2^{\,r_2-1}
\cdot\lambda
\,\neq\, 0.
\end{equation}
Suppose now that \(i\,>\,m\). Then it is clear that
\[
r_1\,+\,i\,-\,m\,-\,1\,\ge\, r_1.
\]
Hence, $\delta_1^{\,r_1+i-m-1}\cdot\lambda\,=\,0,$ as \(\dim\bigl(\mathbb P(E_{1*})\bigr)\,=\,r_1\). It follows that
\[
\phi_i\cdot M_m\,=\,0.
\]
Similarly, if \(i\,<\,m\), then
\[
r_2\,+\,m\,-\,i\,-\,1\,\ge\, r_2,
\]
and consequently
\[
\delta_2^{\,r_2+m-i-1}\cdot\lambda\,=\,0,
\]
since \(\dim\bigl(\mathbb P(E_{2*})\bigr)\,=\,r_2\). Hence,
\begin{equation}\label{m}
\phi_i\cdot M_m \,=\, 0
\quad \text{for all}\, \,i\,\neq\, m.
\end{equation}
It follows from \eqref{i} and \eqref{m} that
\[
\phi_i\cdot M_m\,=\,
\begin{cases}
\delta_1^{r_1-1}\cdot\delta_2^{r_2-1}\cdot\lambda,
& \,i\,=\,m,\\[2mm]
0,
&\, i\,\neq\, m.
\end{cases}
\]
Furthermore,
\[
\psi_j\cdot M_m\,=\,0
\]
for every \(j\), by Corollary \ref{int}. It follows that the classes $\phi_{k'+1},\,\cdots,\,\phi_{n_{11}-1}, \, \psi_{k^{\prime}},\,\cdots,\, \psi_{n_{11-1}}$ are numerically independent. Thus, by \cite[Corollary 3.22]{FL2}, they form the generators of \(\overline{\mathrm{Eff}}_k(S)\) in the case \(k\,=\,n_{11}\,-\,1\,-\,k'\). The proofs of the remaining cases are analogous.
Moreover, by \cite[Theorem 3.3]{Kar}, the embedding $\widetilde{i}$ (see  \eqref{embo}) induces the isomorphism 
\begin{equation}
\overline{\textrm{Eff}}_{k}(\mathbb P(\widetilde Q_{1}^{1})
\,\times_Y\,
\mathbb P(\widetilde Q_{2}^{1}))
\ \longrightarrow\
\overline{\textrm{Eff}}_{k}(\mathbb P(\widetilde E_{1})
\,\times_Y\,\mathbb P(\widetilde E_{2}))
\end{equation}
for all $k\,<\,\textbf{n}$.
On the other hand, we have shown that the pushforward map  $\widetilde{p}_{\ast}$ induces an isomorphism between $\overline{\textrm{Eff}}_{k}(\widetilde{S})\,\cong\,\overline{\textrm{Eff}}_{k}({S})$, for all $k\,\leq\,\textbf{n}$. Moreover, Theorem \ref{stable} identifies the pseudoeffective cones $\overline{\textrm{Eff}}_{k}(\mathbb{P}(Q^{1}_{1\ast})\,\times_{X}\,\mathbb{P}({Q^{1}_{2\ast}}))\,\cong\,\overline{\textrm{Eff}}_{k}(\mathbb{P}(\widetilde{Q}^{1}_{1})\,\times_{Y}\,\mathbb{P}({\widetilde{Q}^{1}}_{2}))\,\,\textrm{for all} \,\, k\,<\,\textbf{n}$. Combining these isomorphisms, we conclude that the embedding $i$ induces an $$i_{\ast}\,:\,\overline{\mathrm{Eff}}_{k}
\bigl(
\mathbb P(Q_{1*}^{1})
\,\times_X\,
\mathbb P(Q_{2*}^{1})
\bigr)
\,\longrightarrow\,
\overline{\mathrm{Eff}}_{k}(S).$$
This completes the proof.
\end{proof}

\begin{remark}
Following the notation of Theorem~\ref{usp}, it remains to determine the pseudoeffective cones
\(\overline{\mathrm{Eff}}_{k}(S)\) for
\(\mathbf n\, \le\, k\, \le\, r_{1}\,+\,r_{2}\,-\,2\).
To do so, one first computes the generators of
\(\overline{\mathrm{Eff}}_{k}(\widetilde S)\)
in the same range. This can be achieved inductively on the lengths of the Harder--Narasimhan filtrations of \(\widetilde E_{1}\) and \(\widetilde E_{2}\), using the isomorphisms established in \cite[Theorem~3.3]{Kar} and \cite[Theorem~3.5]{Kar}. Once the generators of \(\overline{\mathrm{Eff}}_{k}(\widetilde S)\) have been determined, one computes their images under the pushforward map
\[
\widetilde{p}_{\ast}\ :\ \overline{\mathrm{Eff}}_{k}(\widetilde S)\ \longrightarrow\ \overline{\mathrm{Eff}}_{k}(S).
\]
Then it can be shown that these images are precisely the generators of \(\overline{\mathrm{Eff}}_{k}(S)\).
\end{remark}

\section*{Declarations}

On behalf of all authors, the corresponding author states that there is no conﬂict of interest. No data
were used or generated.

\end{document}